\RequirePackage{fix-cm}
\documentclass[pdflatex,sn-mathphys-ay]{sn-jnl}

\usepackage{xcolor}

\usepackage{graphicx,grffile}
\usepackage{amsmath,amssymb,amsthm}
\usepackage{mathtools,dsfont,centernot}
\usepackage{caption}
\usepackage{booktabs,tabularx,threeparttable,longtable,ragged2e}
\usepackage{siunitx}
\usepackage[shortlabels]{enumitem}
\usepackage{alltt}
\usepackage{float} 
\usepackage[nodisplayskipstretch]{setspace}
\usepackage[bottom]{footmisc}
\usepackage{xr} %xr,xr-hyper}

\usepackage{chngcntr}

\usepackage{mathrsfs}
\usepackage[final]{listings}
\lstdefinestyle{inlineR}{language=R,frame=none,basicstyle=\ttfamily,keywordstyle=\ttfamily,stringstyle=\ttfamily,keepspaces=true,showspaces=false,showstringspaces=false,breaklines=true,upquote=true,print,columns=fullflexible}
\newcommand{\code}{\lstinline}

\usepackage[capitalize,noabbrev]{cleveref}
\usepackage{multirow}
  \renewenvironment{thebibliography}[1]%
  {\begin{oldthebibliography}{#1}\setlength{\parskip}{0ex}\setlength{\itemsep}{0ex}}%
  {\end{oldthebibliography}}
\appto\TPTnoteSettings{\linespread{1}\footnotesize}

\newcommand{\citeposs}[1]{\citeauthor{#1}'s (\citeyear{#1})}

\crefname{conjecture}{Conjecture}{Conjectures}
\crefname{section}{Section}{Sections}
\crefname{subsection}{Section}{Sections}
\crefname{subsubsection}{Section}{Sections}
\Crefname{conjecture}{Conjecture}{Conjectures}
\Crefname{section}{Section}{Sections}
\Crefname{subsection}{Section}{Sections}
\Crefname{subsubsection}{Section}{Sections}
\crefname{appendix}{Appendix}{Appendices}
\crefname{subappendix}{Appendix}{Appendices}
\crefname{subsubappendix}{Appendix}{Appendices}
\Crefname{appendix}{Appendix}{Appendices}
\Crefname{subappendix}{Appendix}{Appendices}
\Crefname{subsubappendix}{Appendix}{Appendices}
\crefname{equation}{}{}
\Crefname{equation}{Equation}{Equations}
\crefname{enumi}{}{}
\Crefname{enumi}{}{}
\renewcommand{\theenumi}{\roman{enumi}}

\crefname{assumption}{}{}
\Crefname{assumption}{Assumption}{Assumptions}
\crefname{assumpB}{}{}
\Crefname{assumpB}{Assumption}{Assumptions}
\Crefname{method}{Method}{Methods}
\newlist{steps}{enumerate}{1}
\setlist[steps]{label=\arabic*., ref=\arabic*, itemsep=0pt}
\crefname{stepsi}{Step}{Steps}
\Crefname{stepsi}{Step}{Steps}

\newtheorem{theorem}{Theorem}[section]
\newtheorem{lemma}[theorem]{Lemma}
\newtheorem{proposition}[theorem]{Proposition}
\newtheorem{corollary}[theorem]{Corollary}

\newtheorem{method}{Method}

\Crefname{procedure}{Procedure}{Procedures}

\theoremstyle{definition}
\newtheorem{definition}{Definition}
\newtheorem{assumption}{Assumption}

\newtheorem{assumpC}{Assumption}
\newtheorem{assumpIVQR}{Assumption}

\usepackage{bm} %better bold math for vecf and matf
\newcommand{\ubar}[1]{\mkern3mu\underline{\mkern-3mu #1\mkern-3mu}\mkern3mu}
\newcommand{\matf}[1]{\ubar{\bm{#1}}} %matrix formatting
\newcommand{\vecf}[1]{\bm{#1}} %vector formatting
\newcommand{\numnornd}[1]{\num[round-mode=none,group-digits=integer]{#1}} % no rounding, but group-separator

\newcommand{\iid}{\stackrel{\mathit{iid}}{\sim}}
\newcommand{\pconv}{\xrightarrow{p}}
\newcommand{\dconv}{\xrightarrow{d}}

\DeclareMathOperator{\Var}{Var}

\newcommand{\R}{{\mathbb R}}

\DeclareMathOperator{\E}{E} %{\mathbb{E}}
\let\Pr\relax \DeclareMathOperator{\Pr}{P} %comment out if Pr desired

\DeclareMathOperator*{\plim}{plim}
\newcommand{\plimn}{\plim_{n\to\infty}}

\DeclareMathOperator{\1}{\mathds{1}}
\DeclareMathOperator{\diag}{diag}

\newcommand{\Ind}[1]{\1\{#1\}}
\newcommand{\FWER}{\textrm{FWER}}
\newcommand{\NormDist}{\mathrm{N}}

\newcommand{\UnifDist}{\textrm{Unif}}

\newcommand{\BinoDist}{\textrm{Binomial}}

\providecommand{\abs}[1]{\lvert#1\rvert}

\let\originalleft\left
\let\originalright\right
\renewcommand{\left}{\mathopen{}\mathclose\bgroup\originalleft}
\renewcommand{\right}{\aftergroup\egroup\originalright}

\newcommand{\mockalph}[1]{}  % for BibTeX sorting
\allowdisplaybreaks[3]

\begin{document}

\title{Multiple Testing of a Function's Monotonicity}

\author[1]{\fnm{Wei} \sur{Zhao}}
\author*[1]{\fnm{David M.} \sur{Kaplan}}\email{kaplandm@missouri.edu}
\affil*[1]{\orgdiv{Department of Economics}, \orgname{University of Missouri}, \orgaddress{\street{615 Locust St}, \city{Columbia}, \postcode{65211}, \state{MO}, \country{USA}}}

\abstract{%
Instead of having a single ``yes'' or ``no'' result from a test of the global null hypothesis that a function is increasing, we propose a multiple testing procedure of the function's increasingness at several points.
If the global null is rejected, then multiple testing provides more information about why.
If the global null is not rejected, then multiple testing can provide stronger evidence in favor of increasingness, by rejecting null hypotheses that the function is decreasing.
Our approach uses high-level assumptions that apply to a broad class of causal and descriptive statistical models.
By inverting the proposed multiple testing procedure that controls the familywise error rate, we also generate ``inner'' and ``outer'' confidence sets for the set of points at which the function is increasing.
With high asymptotic probability, the inner confidence set is contained within the true set, whereas the outer confidence set contains the true set.
We also improve power with stepdown and two-stage procedures.
Simulation and empirical examples illustrate the new methodology, and all code is provided.
}

 % Paired and multiple comparisons; multiple testing (...under Linear inference, regression....)


\maketitle

\vspace*{-2\baselineskip}
Date: August 31, 2025\\[\baselineskip]
This version of the article has been accepted for publication, after peer review, but is not the Version of Record and does not reflect post-acceptance improvements or any corrections.
The Version of Record is available online at \url{https://doi.org/10.1007/s11749-025-00986-6}.
Use of this Accepted Version is subject to the publisher's Accepted Manuscript terms of use: \url{https://www.springernature.com/gp/open-research/policies/accepted-manuscript-terms}

\clearpage

\section{Introduction}
\label{introduction}

Unlike a global test of the single null hypothesis that a whole function is monotonically increasing, we propose a multiple testing procedure (MTP) for a function's increasingness at multiple points.
These multiple points correspond to multiple null hypotheses, for which our procedure controls the familywise error rate.
Testing where a function is decreasing is equivalent to testing where the negative of that function is increasing, so we focus on increasingness without loss of generality.

A global hypothesis test with a single null can only provide a ``yes or no'' result.
That is, it can only say either ``yes, we have enough evidence to reject the null that the entire function is increasing'' or ``no, we cannot reject the null that the entire function is increasing.''

In contrast, our MTP can provide more information by testing increasingness at several specific points.
Notationally, we refer to these points as values of ``$X$,'' but $X$ is not restricted to be a regressor; for example, it can represent quantile index values in a quantile regression.
The MTP assesses at which values of $X$ the function is increasing.
For example, imagine a function that is generally increasing except for a small decreasing segment near $X=x$.
The global null hypothesis should be rejected because it is false, but this result arguably misses the big picture that the function is mostly increasing.
Instead, the MTP can reject increasingness near $X=x$ specifically, without rejecting increasingness elsewhere, to provide a more precise assessment.

Beyond non-rejection, the MTP can provide even stronger evidence in favor of increasingness at certain points.
We know that ``non-rejection of a null'' is relatively weak evidence in favor of the null, given the possibility of a type II error (failure to reject false null), whose rate is not controlled.
However, if we reverse the original nulls from ``increasing'' to ``decreasing,'' then rejection provides strong statistical evidence in favor of increasingness at those points.
The MTP may still make an error, but it would be a type I error (rejection of true null), whose rate is controlled at the desired level, in the familywise sense.

Assessing a function's increasingness is a general statistical task with applications across many fields.
We provide a few examples here.
In finance, the below-mentioned work by \citet{RomanoWolf2013} is motivated by theories about financial returns increasing in a variety of asset or portfolio characteristics.
In public health, ``gradients'' encompass the positive or negative relationship between health and various socioeconomic variables, like the health--income gradient; for example, \citet{RidleyEtAl2020} review evidence on the causal relationship between poverty and mental health.
Generally, these gradients can be either descriptive or causal; our methods can apply to either.
In economics, theories related to information costs and rational inattention have implications for monotonicity of certain conditional probabilities, some of which have been tested experimentally; for example, see Section 5.3 of \citet{DewanNeligh2020} and Experiment 1.2 of \citet{DeanNeligh2023}, which both test for the correct choice probability increasing in the incentive level.
In climatology, \citet{FriedrichEtAl2020} test for atmospheric ethane increasing over time during some but not all periods, for which our multiple testing approach could inductively find such periods, rather than running a global test on a pre-specified period.
In ecology, a review paper by \citet{ZhangEtAl2015eco} focuses on non-monotonicity in many ecological settings, where our multiple testing approach can help gauge evidence for such non-monotonicity, including not merely its existence but specifically where in a function it occurs.

Additionally, the MTP can be useful for helping assess identifying assumptions.
For example, \citet{ManskiPepper2000} discuss partial identification of average treatment effects given various monotonicity assumptions, including monotone treatment selection combined with monotone treatment response.
This combination has a testable implication that the conditional expectation function is a weakly increasing function; see their (19) and footnote 9.
For example, in the empirical analysis of the return to schooling, the MTP results would be useful for assessing the combination of monotone treatment selection and monotone treatment response, which jointly imply the expectation of log wage increases with years of schooling.
Instead of using a global hypothesis test and abandoning whole assumptions because they are violated somewhere, the MTP can be used to find subpopulations where the assumptions are met, and then the causal analysis can be restricted to such subpopulations.
More specifically, we could either optimistically use the subpopulations where we do not reject mean log wage increasing with education, or we could more conservatively use the subpopulations where we reject that mean log wage decreases with education in favor of increasingness.
These subpopulations respectively correspond to our outer and inner confidence sets described below.

Inverting the MTP, we also construct ``inner'' and ``outer'' confidence sets.
The object of interest is the true set of points at which the function is increasing.
The inner confidence set is contained within the true set with high asymptotic probability.
Similarly, the outer confidence set contains the true set with high asymptotic probability.
The inner and outer confidence set mechanics come from \citet{Kaplan2024}, who builds on other work to develop them in the context of sets of utility functions.

Theoretically, we provide a unified framework that applies to a wide array of models.
Specifically, we assume the relevant points on the function have estimators that are jointly asymptotically normal, with a consistently estimable covariance matrix.
For example, this applies to many estimators of functions that have a causal interpretation, such as an average structural function \citep[][\S8.1]{BlundellPowell2003}, average index function \citep[][\S8]{LewbelEtAl2012}, or structural quantile function \citep[][\S3.1]{ImbensNewey2009}, which could be estimated by instrumental variables or control function methods.
This also applies to estimators of descriptive models, such as conditional mean and quantile functions.
As noted earlier, the function does not even need to be a function of a regressor, such as looking at a quantile regression slope coefficient as a function of the quantile index.
Our use of the covariance matrix allows our MTP to have asymptotically exact familywise error rate in the ``least favorable'' case of a flat function, whereas the Bonferroni approach generally is conservative.

Further, we offer two ways to improve power.
First, we use a stepdown procedure, like that of \citet{Holm1979}.
Second, we propose two-stage procedure following the approach of \citet{RomanoShaikhWolf2014}, using the least favorable null within a first-stage confidence set.
These two approaches have complementary strengths, and we show simulation results for a data-generating process where their combined power improvement exceeds the sum of their individual improvements.

The main limitation of our MTP is the use of a finite number of $X$ points.
However, this still allows $X$ to be continuous, as long as the number of evaluations points is finite, for example by using deciles.
However, with continuous $X$, the optimal choice of evaluation points remains an open question.
Naturally, the $X$ variable can also be discrete, or even ordinal, in which case potentially all possible values can be used.

Although we are not aware of any existing multiple testing methodology to assess monotonicity, there are related literatures on both monotonicity testing and multiple testing.
In the special case of a conditional mean function, some related papers propose tests of ``regression monotonicity'' with a single null.
For example, \citet{GhosalEtAl2000} and \citet{Chetverikov2019} have the single null ``$H_0\colon m(\cdot)$ is an increasing function'' for the conditional mean function $m(x)=\E(Y\mid X=x)$.
\Citet{HallHeckman2000} design a test statistic for the single null ``$H_0\colon m(\cdot)$ is nondecreasing on the interval $\mathscr{I}$'' without requiring computation of a curve estimator.
\Citet{RomanoWolf2013} also test a single null, the negation of the alternative hypothesis that financial returns are strictly increasing in a given characteristic; see their (2.11).
Alternatively, \citet{KostyshakLuo2021} propose a ``partial monotonicity parameter'' that measures the proportion of the population for which an increase in $X$ is associated with an increase in $Y$, complementing our methods that focus on \emph{where} in the domain the function is increasing.
On the multiple testing side, the general strategy for our basic MTP is similar to that of \citet{WuKaplan2025a}, who instead consider stochastic monotonicity with ordinal or discrete outcomes.
Besides the difference in application to stochastic monotonicity instead of function monotonicity, which entails different equations and formulas (like our \cref{supp_eqn:Zhat-asy-dist} vs.\ their Lemma 1), they do not provide refinements to improve power like in our \cref{sec:power}, and their approach with continuous outcomes uses a different approach.
There is also a literature on testing moment inequalities in general, some of which we apply here; for example, see Section 4 of the survey by \citet{CanayShaikh2017}, as well as the work of  \citet{RomanoShaikhWolf2014} that we adapt to our setting.

\paragraph{Paper structure}
\Cref{sec:MTP,sec:CS} describe new methodology and provide formal theoretical results.
\Cref{sec:emp-CEF,sec:sim} contain empirical and simulation results, respectively.
\Cref{sec:power} provides two additional methods to improve power.
In the supplemental appendix, \cref{supp_sec:app-ext} shows how our methodology can apply to instrumental variables quantile regression as well as functional coefficient models, and \cref{supp_sec:app-proofs} collects proofs.

\paragraph{Notation and abbreviations}
Random and non-random vectors are respectively typeset as, e.g., $\vecf{X}$ and $\vecf{x}$, 
while random and non-random scalars are typeset as $X$ and $x$, 
and random and non-random matrices as $\matf{X}$ and $\matf{x}$.
Acronyms used include those for 
conditional expectation function (CEF), 
confidence set (CS), 
familywise error rate (FWER), 
multiple testing procedure (MTP), 
pointwise rejection probability (PRP), 
and 
two-stage least squares (2SLS).

\section{Multiple Testing}
\label{sec:MTP}

This section presents our MTP and establishes its strong control of asymptotic FWER.

\subsection{Setting}

Consider learning whether or not a function $m(\cdot)$ is increasing.
This $m(\cdot)$ is a function of a scalar $x$, but the underlying model may include interactions between $X$ and other observed and unobserved variables, whose realizations are fixed when defining $m(\cdot)$.

We consider testing at $h$ different values of $X$.
These points are not necessarily the full support of $X$; for example, if $X$ is continuous, we could test at deciles of $X$.
For notational simplicity, these points are written as $x\in \mathcal{X}= \{1,\ldots,h \}$, but more generally $x=j$ can be interpreted as $x=x_j$ for arbitrary $\mathcal{X}=\{x_1,\ldots,x_h\}$.

The MTP decides whether or not to reject each hypothesis under consideration.
Specifically, the MTP tests the following $h-1$ nulls $H_{0x}$ simultaneously.
Defining $d_x\equiv m(x)-m(x+1)$,
\begin{equation}
\label{eqn:H0x}
\begin{split}
& H_{01}\colon d_1 \le 0\quad  \textrm{(the function $m(\cdot)$ is increasing from $x=1$ to $2$)}\\
& H_{02}\colon d_2 \le 0\quad  \textrm{(the function $m(\cdot)$ is increasing from $x=2$ to $3$)}\\
&   ~~\vdots\\
&H_{0,h-1}\colon d_{h-1} \le 0\quad  \textrm{(the function $m(\cdot)$ is increasing from $x=h-1$ to $h$)} .
\end{split}
\end{equation}
The MTP makes $h-1$ decisions, with $2^{h-1}$ possible outcomes.

In \cref{sec:ex-CEF}, we introduce a running example.

% \subsubsection{Example: two-stage least squares (2SLS)}
% \label{sec:ex-2SLS}

% In this example, suppose the causal model is $Y=\beta_0+\beta_1 X+\beta_2 X^2+ \vecf{\beta}_3' \vecf{W}+U$, where $X$ is a scalar endogenous variable of interest and $\vecf{W}$ is a vector of exogenous controls.
% Our function of interest is the contribution of $X$: $m(x)\equiv\beta_1 x + \beta_2 x^2$.
% Given $x\in\mathcal{X}=\{1,2,\ldots,h\}$,
% \begin{equation*}
% \begin{aligned}
% &m(1)=\beta_1 +\beta_2, \;m(2)=2 \beta_1 +2^2 \beta_2, \;\dots, \;m(h)=h \beta_1 +h^2 \beta_2,\\
% &H_{01}\colon d_1=m(1)-m(2)=-\beta_1-3\beta_2\le0,\\
% &H_{02}\colon d_2=-\beta_1-5\beta_2\le0,\\
% & ~~\vdots\\
% &H_{0,h-1}\colon d_{h-1}=-\beta_1-(2h-1)\beta_2 \le 0 .
% \end{aligned}
% \end{equation*}
% More generally than $x\in\{1,\ldots,h\}$, we could compute $m(x)$ at any $x_1<\cdots<x_h$ and write the nulls as
% \begin{equation*}
% H_{0j}\colon d_j\le0
% \textrm{ for }j=1,\ldots,h-1
% ,\quad
% d_j \equiv m(x_j)-m(x_{j+1})
%     = (x_j-x_{j+1})\beta_1 + (x_j^2-x_{j+1}^2)\beta_2
% .
% \end{equation*}
% % 
% The MTP tests these $h-1$ nulls simultaneously.
% The results help us learn whether the function $m(\cdot)$ is increasing or not with respect to $X$ at each point.

\subsubsection{Example: conditional expectation function (CEF)}
\label{sec:ex-CEF}

% In this example, d
Define the conditional expectation function (CEF) as $m(x)\equiv\E(Y\mid X=x)$, $x\in \mathcal{X}=\{1,2,\dots,h\}$.
The null hypotheses are
\begin{equation*}
\begin{aligned}
&H_{01}\colon d_1=m(1)-m(2)=\E(Y\mid X=1)-\E(Y\mid X=2)\le0,\\
&H_{02}\colon d_2=\E(Y\mid X=2)-\E(Y\mid X=3)\le0,\\
& ~~\vdots\\
&H_{0,h-1}\colon d_{h-1}=\E(Y\mid X=h-1)-\E(Y\mid X=h)\le0.
\end{aligned}
\end{equation*}
Again, more generally, we could use any $x_1<\cdots<x_h$ and write the nulls as
\begin{equation*}
\begin{split}
H_{0j}\colon d_j&\le0
\textrm{ for }j=1,\ldots,h-1
,\\%\quad
&d_j \equiv m(x_j)-m(x_{j+1})
    = \E(Y\mid X=x_j) - \E(Y\mid X=x_{j+1})
.
\end{split}
\end{equation*}
The MTP tests these $h-1$ nulls simultaneously.
The results help us learn whether the CEF is increasing or not at each point.

\subsection{Benefits of multiple testing for monotonicity}
\label{sec:MTP-benefits}

There are two main benefits of multiple testing in this context.

First, simply knowing that the global hypothesis test rejected or not does not have as much information as the MTP results.
The global hypothesis test only says whether the whole function $m(\cdot)$ is increasing or not.
The MTP further tells us in which region it is increasing or not.
For example, the global hypothesis test result ``reject $H_0$'' means that the function $m(\cdot)$ is not increasing everywhere in the domain, but that does not provide any information about where the increasing trend is broken.
In contrast, even the simple MTP result ``reject $H_{01}$ but not $H_{02}$'' not only tells us that the function $m(\cdot)$ is not increasing overall, but specifically says there is enough evidence to reject that it increases from $m(1)$ to $m(2)$.

Second, the MTP also can provide strong evidence in favor of $m(\cdot)$ increasing at certain points, by switching the direction of the null hypothesis.
In the global hypothesis test, non-rejection of $H_0\colon m(\cdot)\textrm{ is increasing everywhere}$ is relatively weak evidence in favor of increasingness, and in practice it is rare to have the strength of evidence required to reject the reversed null ``$H_0^*\colon m(\cdot)\textrm{ is not increasing everywhere}$.
Instead, with multiple testing of the reversed nulls
\begin{equation*}
H_{0x}^*\colon d_x\ge0\textrm{ for }x\in\mathcal{X} ,
\end{equation*}
the MTP can provide strong evidence in favor of $m(\cdot)$ increasing at certain points, even if there is not strong evidence at every single point.

\subsection{Assumptions}
\label{sec: assumption}

\Cref{a:asy-normal} is maintained throughout.

\begin{assumption}
\label{a:asy-normal}
The estimator $\hat{\vecf{m}}\equiv(\hat{m}(1),\ldots,\hat{m}(h))'$ is asymptotically normal: given true value $\vecf{m}\equiv(m(1),\ldots,m(h))'$,
\begin{equation*}
% \label{eqn:asy distribution 1}
\sqrt{n}(\hat{\vecf{m}}-\vecf{m})
\dconv  \NormDist(\vecf{0},\matf{V}^a) ,
\end{equation*}
where positive definite matrix $\matf{V}^a$ can be estimated consistently, $\hat{\matf{V}}{}^a\pconv\matf{V}^a$.
\end{assumption}

\Cref{a:asy-normal} has a couple notational simplicities that immediately generalize.
First, the specific $\sqrt{n}$ convergence rate does not affect any proof and can be replaced by $n^r$ for some $r>0$.
For example, this allows nonparametric estimators, where $\vecf{m}$ is a finite-dimensional functional of some underlying population function, but its estimator $\hat{\vecf{m}}$ need not have a $\sqrt{n}$ rate, like if $\vecf{m}$ is the evaluation functional (points on the underlying function); for more discussion and technical results, see Section 3.4 of \citet{Chen2007}.
Second, recall that $m(j)$ can be interpreted as $m(x_j)$, in which case \Cref{a:asy-normal} refers to asymptotic normality of $\hat{\vecf{m}}\equiv(\hat{m}(x_1),\ldots,\hat{m}(x_h))'$ for arbitrary $x_1<\cdots<x_h$.

\Cref{a:asy-normal} can allow for time series data.
In standard cases, the key is to use a covariance matrix estimator that is robust to dependence, such as the class of heteroskedasticity and autocorrelation consistent (HAC) estimators proposed and studied by \citet{NeweyWest1987} and \citet{Andrews1991HAC} and implemented in the R package \code{sandwich} \citep{R.sandwich}.
For example, to estimate the CEF $\E(Y\mid X=x)$ with discrete or ordinal $X$, we can regress $Y$ on indicator variables for each $X$ value and no constant, and given assumptions on finite moments, stationarity, dependence, and the kernel function and lag length, Theorem 1(a) of \citet{Andrews1991HAC} establishes consistency of this class of covariance matrix estimators.
Conversely, of course, if there is nonstationarity or dependence is too strong, then our \Cref{a:asy-normal} does not hold.
For example, with a first-order autoregressive AR(1) process $y_{t}=\phi y_{t-1}+\varepsilon_{t}$, where $\varepsilon_{t}$ is white noise with variance $\sigma^2$, the OLS estimator $\hat{\phi}$ is not asymptotically normal if there is a unit root, $\phi=1$.
An interesting special case is when we wish to learn if the mean of a time series is increasing over time, i.e., monotonicity of $\E(Y_t)$ in $t$.
We can partition the time series sample $\{Y_t\}_{t=1}^{T}$ into $h$ blocks of consecutive observations and compute the sample mean of each block; equivalently, we can regress $Y_t$ on the $h$ block indicator variables and no constant, the corresponding estimated coefficients of which are $\hat{m}(j)$ for $j=1,\ldots,h$.
Asymptotically, if we consider a fixed number of observations per block and thus growing number of blocks $h$, then generally \Cref{a:asy-normal} does not hold; for example, the block averages $\hat{m}(j)$ are not normally distributed, unless the $Y_t$ are Gaussian already.
Alternatively, if we have a fixed number of blocks $h$ and growing number of observations per block, then \Cref{a:asy-normal} may indeed hold.
In practice, \citet{IbragimovMueller2010} emphasize that we should think about which asymptotic approximation seems more appropriate for our particular sample, given the number of observations per block, for example.
They further consider when blocks are large enough that they are approximately independent, so that cross-block correlations are negligible.
This also depends on the empirical setting; for example, a climate analysis may have significant dependence across decades, whereas stock market returns may have negligible serial correlation at even a daily sampling frequency.

Instead of \Cref{a:asy-normal}, we could make the high-level assumption that the asymptotic distribution of $n^r(\hat{\vecf{m}}-\vecf{m})$ exists and can be consistently estimated by a bootstrap or subsampling.
In that case, the critical value can be computed as in \cref{sec:cv-bootstrap}, which can also provide some finite-sample improvement even under \Cref{a:asy-normal}.
Thus, our approach can be applied in certain non-normal cases, too.

\Cref{a:asy-normal} is a relatively high-level assumption that holds in a wide variety of settings.
Below, we continue our example from \cref{sec:ex-CEF} to show how to establish this high-level assumption from lower-level assumptions.

\subsubsection{Example: CEF (continued)}   

Continuing from \cref{sec:ex-CEF}, in the CEF example of $m(x)=\E(Y\mid X=x)$ for $x\in\{1,\ldots,h\}$, the following low-level assumptions are sufficient for \Cref{a:asy-normal}.

\begin{assumpC}
\label{a:CEF}
The subsample sizes are $n_x\equiv\sum_{i=1}^{n}\Ind{X_i=x}$, and they are fixed, with iid sampling from the respective conditional distributions.
For example, there are $n_1$ iid draws from the conditional distribution of $Y$ given $X=1$, $n_2$ iid draws from the conditional distribution of $Y$ given $X=2$, and generally $n_x$ iid draws from the conditional distribution of $Y$ given $X=x$.
Also, $n_1/n_x\to\gamma_x\in(0,\infty)$ for all $x \in \mathcal{X}$.
\end{assumpC}

Under \Cref{a:CEF}, using the central limit theorem and continuous mapping theorem \citep[e.g.,][Prop.\ 2.27 and Thm.\ 2.3]{vanderVaart1998},
the asymptotic distribution of $\sqrt{n_1} (\hat{\vecf{m}}-\vecf{m})$ is
\begin{equation}
\begin{aligned}
\label{eqn:CEF-mhat-asy-dist}
\sqrt{n_1} (\hat{\vecf{m}}-\vecf{m})
\dconv  \NormDist(\vecf{0},\matf{V}^a)
\end{aligned}
\end{equation}
as in \Cref{a:asy-normal}, where $\vecf{0}\equiv(0,\dots,0)'$, $V^a_{(xj)}= \gamma_x v_x \Ind{x=j}$ represents the row $x$, column $j$ element of diagonal matrix $\matf{V}^a$ for $x,j\in\mathcal{X}$, and $v_x\equiv\Var(Y\mid X=x)$.

\subsection{Multiple testing procedure}
\label{sec:plain-MTP}

Below, we describe our MTP and state its asymptotic FWER control.

Our MTP controls the ``overall type I error rate'' known as the familywise error rate (FWER).
This is one way to quantify the false positive rate that we want to control when looking across a family of hypotheses.
There are several alternatives to FWER, like $k$-FWER and false discovery proportion \citep{LehmannRomano2005fwer}, as well as the false discovery rate \citep{BenjaminiHochberg1995}, but such are left to future work.
One benefit of FWER control is that it justifies inversion of our MTP into confidence sets as in \cref{sec:CS}.
\Cref{def:FWER,def:FWER-control} follow \citet[\S9.1, p.\ 407]{LehmannRomano2022text}, jointly defining our MTP's desired property of strong control of asymptotic FWER.

\begin{definition}[familywise error rate]
\label{def:FWER}
For a family of null hypotheses $H_{0x}$ indexed by $x$, let $\mathcal{T}\equiv\{x:H_{0x}\textrm{ is true}\}$ be the set of indices of true hypotheses.
The FWER is the probability of rejecting any true null hypothesis.
Mathematically,
\begin{equation}
\label{eqn:FWER}
\FWER \equiv
\Pr(\text{reject any }H_{0x}\text{ with }x \in \mathcal{T}) .
\end{equation}
\end{definition}

\begin{definition}[strong control of FWER]
\label{def:FWER-control}
Strong control of FWER requires $\FWER\le\alpha$ for any combination of true and false $H_{0x}$.
Strong control of asymptotic FWER instead requires $\FWER\le\alpha+o(1)$.
\end{definition}

\begin{method}
\label{meth:plain-MTP}
Given the $H_{0x}\colon d_x\le0$ in \cref{eqn:H0x}, the MTP rejects $H_{0x}$ when $\hat{t}_x>c_\alpha$, where the $t$-statistics are $\hat{t}_x=\hat{d}_x/\hat{s}_x$, the estimated standard errors $\hat{s}_x$ are the square roots of the diagonal elements of $\hat{\matf{V}}{}^b/n$, where $\hat{\matf{V}}{}^b$ is a consistent estimator for the $\matf{V}^b$ defined in \cref{supp_eqn:dhat-asy-dist}, and $c_\alpha$ is the critical value.
Specifically, $c_\alpha$ is the $(1-\alpha)$-quantile of the asymptotic distribution of $\hat{Z}^*$, which is the maximum of $h-1$ correlated random variables $\hat{Z}_x$:
\begin{equation*}
\hat{Z}^*\equiv\max_{x\in\{1,2,\dots,h-1\}} \hat{Z}_x
,\quad
\hat{Z}_x \equiv \frac{\hat{d}_x-d_x }{\hat{s}_x} \textrm{ for }x\in\{1,2,\dots,h-1\}
,
\end{equation*}
and the asymptotic joint normal distribution of the $\hat{Z}_x$ is in \cref{res:asy-Zhat}.
Details of simulating $c_\alpha$ are in \cref{sec:cv-sim,sec:cv-bootstrap}.
To test $H_{0x}^*\colon d_x\ge0$, simply run the above MTP for $H_{0x}\colon d_x\le0$ after replacing $Y$ with $-Y$ in the data.
\end{method}

\begin{lemma}
\label{res:asy-Zhat}
Under \Cref{a:asy-normal}, the asymptotic distribution of $\hat{\vecf{Z}}\equiv ( \hat{Z}_1, \dots, \hat{Z}_{h-1} )$ is multivariate normal with mean zero and covariance matrix $\matf{\Sigma}$ defined in \cref{supp_eqn:Zhat-asy-dist}:  $\hat{\vecf{Z}} \dconv \NormDist (\vecf{0},\matf{\Sigma})$. 
\end{lemma}

\begin{theorem}
\label{res:plain-MTP}
Under \Cref{a:asy-normal}, \cref{meth:plain-MTP} has strong control of asymptotic FWER.
\end{theorem}

Although the null hypothesis is very different, \cref{meth:plain-MTP} can be compared to the global monotonicity test in Section 3.1 of \citet{RomanoWolf2013}.
(Their Section 3.3 method further provides a two-step procedure analogous to our \cref{meth:RSW}.)
The consider a more specific setting where $m(j)$ is the expected financial return of asset or portfolio $j$, and $\hat{m}(j)$ is a time series average of returns, but their approach readily generalizes to our setting.
Translating to our setting and notation, their (2.11) states their null and alternative hypotheses $H_0\colon \max_x d_x \ge 0$ vs.\ $H_1\colon \max_x d_x < 0$; that is, $H_1$ is that the function is strictly increasing, and $H_0$ is the negation of $H_1$.
Their Section 3.1 test rejects $H_0$ in favor of $H_1$ when $\max_x \hat{t}_x<z_\alpha$, the $\alpha$-quantile of the standard normal distribution, like $z_{0.05}\approx -1.64$.
This is more closely related to our MTP for the reversed nulls $H_{0x}^*\colon d_x\ge0$ from \cref{sec:MTP-benefits}, which rejects $H_{0x}^*$ when $\hat{t}_x < -c_\alpha$.
In the trivial case of $h=2$, $-c_\alpha=z_\alpha$, so both methods reject when $\hat{t}_1 < -c_\alpha$.
With $h>3$, there are a few different cases.
If $\max_x\hat{t}_x < -c_\alpha$, then the \citet{RomanoWolf2013} test rejects in favor of a strictly increasing $m(\cdot)$, and our method rejects each null in favor of $m(\cdot)$ strictly increasing between every pair of consecutive points.
If $-c_\alpha < \max_x\hat{t}_x < z_\alpha$, then their global test concludes that the function is strictly increasing everywhere, but there is at least one point where our MTP cannot reject in favor of strictly increasing.
In the most extreme case, $-c_\alpha < \min_x\hat{t}_x < \max_x\hat{t}_x < z_\alpha$, so none of the $H_{0x}^*$ are rejected by the MTP in favor of strict increasingness, yet the global test can reject the global null in favor of $m(\cdot)$ strictly increasing everywhere, which is equivalent to all $H_{0x}^*$ being false.
This illustrates that for the global null, the global test of \citet{RomanoWolf2013} has a power advantage over multiple testing.
Conversely, if $\min_x\hat{t}_x < -c_\alpha < z_\alpha < \max_x\hat{t}_x$, then their global test cannot reject in favor of increasingness, but our MTP can reject in favor of increasingness at certain $x$ values.
In one extreme example, imagine $\hat{t}_1=0$ but $\hat{t}_x < -c_\alpha$ for all $x>1$: then our MTP rejects in favor of $m(\cdot)$ strictly increasing over all $x\ge2$, but the global test merely fails to reject the global null of not-everywhere-increasing.%
\footnote{Similar conclusions should apply to the analogous comparison of the global test of \citet{DavidsonDuclos2013} taking restricted first-order stochastic dominance as the alternative hypothesis with the MTP of \citet{GoldmanKaplan2018c}, where instead of $H_{0x}\colon m(x)-m(x+1)\le0$ over discrete $x$ they have $H_{0x}\colon F_1(x)-F_2(x)\le0$ over $x\in\R$ for CDFs $F_1(\cdot)$ and $F_2(\cdot)$, among other variations.}
So, the global test and multiple testing have complementary strengths in power.

The above begs the question: can the global test be combined with the closure method to make a higher-powered MTP?
The answer is no, for the following reason.
In order to reject $H_{0j}^*$ for a particular $j$, the closure method requires the global test to jointly reject every possible combination of $H_{0x}^*$ hypotheses including the particular $H_{0j}^*$.
This includes the joint test of all $H_{0x}^*$.
However, this joint test fails to reject when $\max_x\hat{t}_x > z_\alpha$, i.e., when even a single $t$-statistic is higher than the $z_\alpha$ critical value.
Thus, a necessary condition for \emph{any} $H_{0x}^*$ to be rejected by this closure method is for $\max_x\hat{t}_x < z_\alpha$, so the MTP reduces to ``reject all $H_{0x}^*$'' when $\max_x\hat{t}_x<z_\alpha$, and ``do not reject any $H_{0x}^*$'' otherwise.
Note that this inability to leverage the closure method is due to the increasing-everywhere being the alternative hypothesis; if it were the null hypothesis, and the joint test statistic is the minimum of the individual $t$-statistics, then applying the closure method simply reproduces our MTP in \cref{meth:plain-MTP}.

\subsubsection{Critical value: multivariate normal simulation}
\label{sec:cv-sim}

The critical value $c_\alpha$ must be simulated because the analytic formulation of the asymptotic distribution of $\hat{Z}^*$ is intractable.
\Cref{meth:plain-MTP} says $c_\alpha$ is the $(1-\alpha)$-quantile of the asymptotic distribution of $\hat{Z}^*$, which is a distribution of the maximum of $h-1$ correlated normal random variables with mean zero and covariance matrix $\matf{\Sigma}$.
\Citet{NadarajahKotz2008} derive the exact distribution of the max of two Gaussian random variables, but it is very complex and in practice usually there are more than two.

To simulate $c_\alpha$, we can approximate the asymptotic distribution of $\hat{Z}^*$ using the following method.
\begin{method}[simulated normal critical value]
Run the following steps.
\begin{steps}
 \item For $b=1,\ldots,B$, draw $\vecf{Z}^{(b)}\iid\NormDist(\vecf{0},\hat{\matf{\Sigma}})$, where the covariance matrix estimator is defined just after \cref{supp_eqn:Zhat-asy-dist}.
  \item Given each $\vecf{Z}^{(b)}=(Z^{(b)}_1,\ldots,Z^{(b)}_{h-1})$, compute the maximum $Z^{(b)}_\text{max}\equiv\max_{x\in\{1,\ldots,h-1\}}Z^{(b)}_x$.
 \item The simulated critical value $\hat{\hat{c}}_{\alpha}$ is the $(1-\alpha)$-quantile among the $B$ values of $Z^{(b)}_\text{max}$.
\end{steps}
\end{method}

The first ``hat'' on $\hat{\hat{c}}_{\alpha}$ represents the use of estimated covariance matrix $\hat{\matf{\Sigma}}$, whose estimation error converges in probability to zero as the sample size increases.
The second ``hat'' is because it is simulated.
The simulation error can be made arbitrarily small by taking a large enough number of simulation replications.

To choose the number of replications $B$ in practice, we have the following suggestion.
Specifically, we solve for the smallest $B$ such that there is a high probability $1-\epsilon$ that FWER does not exceed $\alpha+\delta$, abstracting away from potential estimation error in $\hat{\matf{\Sigma}}$ and asymptotic approximation error in the normal distribution.
For example, if $\epsilon=0.02$ and $\delta=0.01$, then we want a $98\%$ probability that FWER is at most one percentage point larger than the desired $\alpha$.
Let $c_{\alpha+\delta}$ be the critical value that achieves exactly $\alpha+\delta$ FWER under the least favorable null.
The simulated critical value will be at least $c_{\alpha+\delta}$ when no more than $(1-\alpha)R$ of the $Z^{(b)}_\text{max}$ are below $c_{\alpha+\delta}$.
Because the $Z^{(b)}_\text{max}$ are iid across $b=1,\ldots,B$, and the true probability of $Z^{(b)}_\text{max}\le c_{\alpha+\delta}$ is $1-\alpha+\delta$ under the least favorable null by definition of $c_{\alpha+\delta}$, then $\sum_{b=1}^{B}\Ind{Z^{(b)}_\text{max}\le c_{\alpha+\delta}} \sim \BinoDist(B,1-\alpha+\delta)$, a binomial distribution that depends only on the user-chosen values $B$, $\alpha$, and $\delta$.
Thus, to ensure a probability of at least $1-\epsilon$ that the FWER distortion due to simulation error does not exceed $\delta$, we can solve for the smallest $B$ that satisfies
\begin{equation*}
\Pr\bigl( \BinoDist(B,1-\alpha-\delta) \le (1-\alpha)B \bigr)
\ge 1-\epsilon .
\end{equation*}
For example, using $\alpha=0.05$, $\delta=0.01$, and $\epsilon=0.02$ yields $R=2180$; decreasing to $\epsilon=0.01$ yields $R=2820$.
Our simulations with $\alpha=0.05$ use $B=10^5$, which is near the solution $B=10{,}800$ when $\delta=0.005$ and $\epsilon=0.01$.

\subsubsection{Critical value: nonparametric bootstrap}
\label{sec:cv-bootstrap}

Alternatively, the critical value can be simulated using a bootstrap or subsampling method.
To give a concrete example, here we describe a Studentized nonparametric bootstrap appropriate for iid data.
We use this method in the simulation in \cref{sec:sim}.

\begin{method}[bootstrap critical value]
Run the following steps.
\begin{steps}
 \item Let $\hat{d}_x$, $\hat{s}_x$, and $\hat{t}_x=\hat{d}_x/\hat{s}_x$ ($x=1,\ldots,h-1$) denote the estimates, standard errors, and $t$-statistics from the original sample of observations $\vecf{W}_i$ ($i=1,\ldots,n$).
 \item\label{step:cv-bs-resample} From the original sample, resample $\vecf{W}^*_i$ ($i=1,\ldots,n$) with replacement.
 \item\label{step:cv-bs-t} Using the bootstrap sample, compute estimates, standard errors, and $t$-statistics $\hat{m}^*_x$ ($x=1,\ldots,h$), $\hat{d}^*_x=\hat{m}^*_x-\hat{m}^*_{x+1}$, $\hat{s}^*_x$ (estimated standard error of $\hat{d}^*_x$), and $\hat{t}^*_x=(\hat{d}^*_x-\hat{d}_x)/\hat{s}^*_x$ ($x=1,\ldots,h-1$), as well as the maximum $t$-statistic $T^*=\max_x\hat{t}^*_x$.  (Note $\hat{t}^*_x$ is centered at $\hat{d}_x$, not zero, where in the ``bootstrap world'' $\hat{d}_x$ is the true population value; $\hat{Z}^*_x$ is perhaps better notation than $\hat{t}^*_x$ but would partially conflict with our earlier $\hat{Z}$ notation.)
 \item Repeat \cref{step:cv-bs-resample,step:cv-bs-t} $B$ times, to get $T^*_b$ for $b=1,\ldots,B$.  (In \cref{sec:sim}, we use $B=1000$.)
 \item The bootstrap critical value $\hat{c}_\alpha$ is the $(1-\alpha)$-quantile of the $T^*_b$ values.
 \item As in \cref{meth:plain-MTP}, reject $H_{0x}$ when $\hat{t}_x>\hat{c}_\alpha$, where $\hat{t}_x$ is from the original sample.
\end{steps}
\end{method}

% \subsubsection{Examples: 2SLS and CEF (continued)}
\subsubsection{Example: CEF (continued)}

% In both the 2SLS and CEF examples, we have asymptotic normality of $\sqrt{n} (\hat{\vecf{m}}-\vecf{m})$ in \cref{eqn:2SLS-mhat-asy-dist,eqn:CEF-mhat-asy-dist}, respectively, which in turn implies asymptotic normality of $\hat{\vecf{Z}}$.
In the CEF example, we have asymptotic normality of $\sqrt{n} (\hat{\vecf{m}}-\vecf{m})$ in \cref{eqn:CEF-mhat-asy-dist}, which in turn implies asymptotic normality of $\hat{\vecf{Z}}$.
By sampling vectors from a mean-zero multivariate normal distribution with the appropriate estimated covariance matrix $\hat{\matf{\Sigma}}$, the simulated critical value is the $(1-\alpha)$-quantile among the simulated maxima.

\section{Confidence Sets}
\label{sec:CS}

Each confidence set (CS) that we propose is for a set-valued parameter, specifically the set of points $x$ at which $m(x)$ is increasing:
\begin{equation}
\label{eqn:true-S}
\mathcal{S}
\equiv \{ x : d_x \le 0 \}
= \{ x : m(x)-m(x+1) \le 0 \} .
\end{equation}

\subsection{Outer confidence set}
\label{outer confidence set}

An outer CS $\hat{\mathcal{S}}_o$ is the more common type of CS, containing the true set $\mathcal{S}$ with high asymptotic probability.
Mathematically,
\begin{equation}
\label{eqn:outer-CS}
\Pr(\hat{\mathcal{S}}_o\supseteq\mathcal{S})\ge1-\alpha+o(1) .
\end{equation}
Faced with uncertainty, an outer CS must err toward being too big and including some $x$ that are not in fact in $\mathcal{S}$.
Put differently, the outer CS must have strong evidence in order to exclude a certain $x$ from $\hat{\mathcal{S}}_o$, otherwise it will violate \cref{eqn:outer-CS}.

\Cref{meth:outer-CS} constructs an outer CS by inverting the MTP in \cref{meth:plain-MTP}, similar to Method 3 of \citet{Kaplan2024} and following the same logic as the (outer) confidence set for the identified set developed by \citet[Lem.~2.1]{RomanoShaikh2010}.
Its asymptotic validity is then given by \cref{res:outer-CS}.

\begin{method}[outer CS]
\label{meth:outer-CS}
Run the MTP in \cref{meth:plain-MTP} with $H_{0x}\colon d_x \le 0$ and let the outer CS be $\hat{\mathcal{S}}_o \equiv \{x : \textrm{MTP does not reject}\ H_{0x} \}$.
\end{method}

\Cref{meth:outer-CS} says that the outer CS collects all $x$ for which the MTP does not reject the corresponding null $H_{0x}$.
Thus, the true set will be a subset of the outer CS as long as none of the true null hypotheses is rejected.
As $1-\alpha$ decreases and thus $\alpha$ increases, each $H_{0x}$ becomes more likely to be rejected, so the outer CS becomes smaller.

\begin{corollary}
\label{res:outer-CS}
Given the conclusion of \cref{res:plain-MTP} and true set $\mathcal{S}$ in \cref{eqn:true-S}, the outer CS $\hat{\mathcal{S}}_o$ in \cref{meth:outer-CS} is asymptotically valid in the sense of \cref{eqn:outer-CS}.
\end{corollary}

\subsection{Inner confidence set}
\label{inner confidence set}

An inner CS $\hat{\mathcal{S}}_i$ is contained within the true set $\mathcal{S}$ with high asymptotic probability.
This idea seems to originate in (1) of \citet{ArmstrongShen2023}, and we generally follow the approach of \citet[][\S\S3,5.3]{Kaplan2024}.
Mathematically,
\begin{equation}
\label{eqn:inner-CS}
\Pr(\hat{\mathcal{S}}_i\subseteq\mathcal{S})
\ge 1 - \alpha + o(1) .
\end{equation}
Faced with uncertainty, the inner CS does the opposite of the outer CS, erring toward being too small and possibly omitting some $x$ that actually are in $\mathcal{S}$.
That is, the inner CS must have strong evidence in order to \emph{include} an $x$ value in $\hat{\mathcal{S}}_i$, otherwise it will violate \cref{eqn:inner-CS}.
Thus, the inner CS can be viewed as a more conservative ``estimate'' of $\mathcal{S}$ in the sense that it only includes $x$ values where there is strong evidence that $m(x)$ is increasing.

\Cref{meth:inner-CS} constructs an inner CS by inverting the MTP in \cref{meth:plain-MTP} with the null hypothesis inequality directions reversed, again following Method 3 of \citet{Kaplan2024}; its asymptotic validity is then given by \cref{res:inner-CS}.

\begin{method}[inner CS]
\label{meth:inner-CS}
Run the MTP in \cref{meth:plain-MTP} with $H_{0x}^*\colon d_x \ge 0$ (i.e., replacing $Y$ with $-Y$) and let the inner CS be $\hat{\mathcal{S}}_i \equiv \{x : \textrm{MTP rejects }H_{0x}^* \}$.
\end{method}

\Cref{meth:inner-CS} says that the inner CS collects all $x$ for which the MTP rejects the corresponding reversed null hypothesis $H_{0x}^*\colon d_x \ge 0$ in favor of increasingness at point $x$, so the inner CS equals the complement of the outer CS for $\{x:m(x)\le m(x+1)\}$, the set of points where $m(x)$ is decreasing.
Intuitively, strong evidence in favor of increasingness is the same as strong evidence against decreasingness, so the set of $x$ with strong evidence in favor of increasingness is the complement of the set of $x$ lacking strong evidence against decreasingness; that is, the inner CS for increasing points is the complement of the outer CS for decreasing points.
This symmetry is reassuring.
Regardless of whether we frame our inquiry in terms of increasingness or decreasingness, the confidence sets essentially partition the $x$ points into three groups: one with strong evidence in favor of increasingness, one with strong evidence in favor of decreasingness, and one without strong evidence in either direction.

As with the outer CS, the coverage of the inner CS is closely linked to the MTP.
Because the inner CS collects $x$ for which the MTP rejects $H_{0x}^*\colon d_x \ge 0$, the probability of any false rejection is the probability of incorrectly including an $x$ in the inner CS.
Thus, asymptotically, strong control of FWER implies correct coverage probability.
Also, as $1-\alpha$ decreases and thus $\alpha$ increases, each $H_{0x}^*$ becomes more likely to be rejected, so the inner CS becomes larger.
That is, as $\alpha$ increases, both the inner CS and outer CS approach the point estimate $\hat{\mathcal{S}}=\{x:\hat{d}_x\le0\}$, but the inner CS approaches from the inside whereas the outer CS approaches from the outside, with $\hat{\mathcal{S}}_i\subseteq\hat{\mathcal{S}}\subseteq\hat{\mathcal{S}}_o$.

\begin{corollary}
\label{res:inner-CS}
Given the conclusion of \cref{res:plain-MTP} and true set $\mathcal{S}$ in \cref{eqn:true-S}, the inner CS $\hat{\mathcal{S}}_i$ in \cref{meth:inner-CS} is asymptotically valid in the sense of \cref{eqn:inner-CS}.
\end{corollary}

\section{Empirical illustration: earnings and education}
\label{sec:emp-CEF}

Economists and others have long studied the relationship between earnings and years of education.
Here we consider not a causal relationship but a descriptive one, the conditional expectation function (CEF).
Further, as noted in \cref{introduction}, this CEF can also be used to assess the assumptions of \citet{ManskiPepper2000} for learning about the causal effect of education on earnings.

\subsection{Setup}
\label{sec:emp-CEF-setup}

In this example, $m(x)=\E(Y\mid X=x)$, and interest is in where log weekly income ($Y$) is increasing in years of education ($X$).
We use the dataset \code{census2000} in package \code{wooldridge} \citep{R.wooldridge}, originally provided by \citet{Wooldridge2010}.
It contains $\numnornd{29501}$ observations.
In the dataset, our $Y$ variable is \code{lweekinc} and our $X$ variable is \code{educ}.

In \citeposs{ManskiPepper2000} empirical example, they compute a uniform $95\%$ confidence band for the CEF $m(\cdot)$.
Then, they claim that the monotone treatment selection and monotone treatment response assumptions are consistent with the empirical evidence because the band includes a weakly increasing CEF.
However, such a band could include an increasing function even if the entire estimated CEF is decreasing.
The MTP can provide a stronger and more specific assessment.

\subsection{Results}

\begin{table}[htbp]
\centering
\caption{\label{tab:emp-CEF}CEF example, $\alpha=0.05$}
\sisetup{round-precision=3}
% \begin{threeparttable}
\begin{tabular}[c]{
r
S[table-format=1.2,round-precision=2]
S[table-format=5.0,round-precision=0]
S[table-format=1.2,round-precision=2]
S[table-format=-2.2,round-precision=2]
cc}
\toprule
  \multicolumn{1}{c}{$x$} & \multicolumn{1}{c}{$\hat{m}(x)$}
& \multicolumn{1}{c}{$n_x$} & \multicolumn{1}{c}{$\hat{\hat{c}}_\alpha$} 
& \multicolumn{1}{c}{$\hat{t}_x$} & \multicolumn{1}{c}{$H_{0x}\colon m(x)\le m(x+1)$} 
& \multicolumn{1}{c}{$H_{0x}^*\colon m(x)\ge m(x+1)$} \\
\midrule
 9 &    6.2586 &    374 & 2.3952 &  -1.1968 & Not reject & Not reject \\
10 &    6.3104 &    621 & 2.3952 &   1.6616 & Not reject & Not reject \\
11 &    6.2475 &    601 & 2.3952 &  -9.7814 & Not reject &     Reject \\
12 &    6.4991 &  12433 & 2.3952 & -12.0561 & Not reject &     Reject \\
13 &    6.6327 &   5424 & 2.3952 &  -1.3792 & Not reject & Not reject \\
14 &    6.6552 &   2625 & 2.3952 & -17.7031 & Not reject &     Reject \\
16 &    6.9397 &   7423 & &&&\\
\bottomrule
\end{tabular}
% \end{threeparttable}
\end{table}

\Cref{tab:emp-CEF} includes the following columns.
The $x$ is the years of education.
The $\hat{m}(x)$ column shows the corresponding CEF point estimates, and $n_x$ is the number of observations with $X_i=x$.
The $\hat{\hat{c}}_\alpha$ is the simulated critical value; again, there is just one, repeated across rows for convenience.
The $\hat{t}_x$ column shows the $t$-statistics defined in \cref{meth:plain-MTP}, with more negative values indicating stronger evidence of $m(\cdot)$ decreasing from $x$ to $x+1$, and more positive values indicating stronger evidence of $m(\cdot)$ decreasing.
The $H_{0x}\colon m(x)\le m(x+1)$ column shows the MTP results for the null hypotheses that mean income is increasing in education from $x$ to $x+1$.
Similarly, the $H_{0x}^* \colon m(x)\ge m(x+1)$ column shows the MTP results for the reversed null hypotheses that mean income is decreasing in education from $x$ to $x+1$.
The testing is done separately for each column, so FWER is controlled within each column but not across both columns simultaneously.

\Cref{tab:emp-CEF} shows the following results.
From the $H_{0x}\colon m(x)\le m(x+1)$ column, we see the MTP does not anywhere reject that mean log weekly income is increasing with education.
(The estimated function actually decreases from $x=10$ to $x=11$, but not enough for even a pointwise $t$-test to reject at a $5\%$ level.)
However, we know that ``not rejecting'' is relatively weak evidence.
Switching the direction of null hypotheses produces stronger evidence in favor of increasingness, seen in the last column of the table.
Specifically, the MTP rejected $H_{0x}^*$ in favor of increasingness for $x\in\{11,12,14\}$.
That is, the MTP tells us there is strong evidence in favor of average log weekly income increasing with education for most values above $x\ge11$.
For example, rejecting $H_{0,12}^*$ says that the mean log weekly income is statistically significantly higher for individuals with $x=13$ years of education than individuals with $x=12$ years of education; even accounting for the fact that we are making several such comparisons simultaneously, there is a statistically significant increase associated with even just one year of college.
The one exception among $x\ge11$ is that $H_{013}^*$ is not rejected, even though the relevant subsample sizes $n_{13}$ and $n_{14}$ are relatively large.
The point estimates $\hat{m}(13)=6.63$ and $\hat{m}(14)=6.66$ are too close to reject $H_{013}^*\colon m(13)\ge m(14)$.

\Cref{tab:emp-CEF} also has enough information to compute the global test of \citet{RomanoWolf2013} as well as a Bonferroni MTP.
Given the $\hat{t}_x$ values, the MTP results will be identical for any critical value between $\hat{t}_{10}=1.66$ and $\abs{\hat{t}_{11}}=9.78$, so the Bonferroni MTP gets the same results.
Similar to the simulation results in \cref{tab:sim-plain}, the Bonferroni critical value is very similar to ours in this dataset, differing by only $0.0012$.
% qnorm(1-0.05/6) = 2.3940
For the test of \citet{RomanoWolf2013}, following the discussion after our \cref{res:plain-MTP} and using $\alpha=0.05$, the test rejects in favor of the alternative that $m(\cdot)$ is increasing when $\max_x\hat{t}_x<z_\alpha=-1.64$.
In \cref{tab:emp-CEF}, $\max_x\hat{t}_x=1.66>0$, so their test does not come close to rejecting.
This is reasonable: $\hat{m}(10)>\hat{m}(11)$, so we could not possibly claim that there is strong evidence in favor of $m(\cdot)$ increasing at every $x$.
However, it also misses our more nuanced MTP results that show strong evidence of the CEF increasing at certain $x$ values that include a large proportion of individuals in the sample.

\Cref{tab:emp-CEF} can also be used to construct inner and outer confidence sets for the true set of points $\mathcal{S} \equiv \{x\colon m(x)\le m(x+1)\}$ where mean log weekly income is increasing in education.
Following \cref{meth:outer-CS,meth:inner-CS}, the outer CS $\hat{\mathcal{S}}_o$ collects all $x$ for which the MTP does not reject the null $H_{0x}$ of increasingness, and the inner CS $\hat{\mathcal{S}}_i$ collects all $x$ for which the MTP rejects the reversed null $H_{0x}^*$ of decreasingness.
\Cref{tab:emp-CEF} shows that the MTP does not reject any $H_{0x}$ and rejects $H_{0x}^*$ for $x\in\{11,12,14\}$, so the confidence sets are
\begin{equation*}
\hat{\mathcal{S}}_o= \{9,10,11,12,13,14\}
,\quad
\hat{\mathcal{S}}_i=\{11,12,14\}
.
\end{equation*}
There is a high probability of sampling a dataset in which the true $\mathcal{S}$ is ``between'' the inner and outer CSs, which here would mean $\{11,12,14\}\subseteq\mathcal{S}\subseteq\{9,10,11,12,13,14\}$, suggesting the CEF increases at least at $11$, $12$, and $14$ years, and possibly at additional values, but there is not strong enough evidence to say either way.

For the monotone treatment selection and monotone treatment response assumptions assessment in \citet{ManskiPepper2000}, the results from \cref{tab:emp-CEF} indicate strong evidence that the testable implication of an increasing CEF is satisfied for points $x=11,12,14$.
For the rest of the points, there is not enough evidence to say for sure that it is satisfied, but also not enough evidence to reject it, which would then imply rejecting one or both of their assumptions.
For an empirical analysis relying on monotone treatment response and selection, this might suggest that we use the full dataset ($\hat{\mathcal{S}}_o$) for a main analysis, but restrict to $x\in\hat{\mathcal{S}}_i=\{11,12,14\}$ for a robustness check.

\section{Simulation with \texorpdfstring{\cref{meth:plain-MTP}}{Method \ref{meth:plain-MTP}} MTP}
\label{sec:sim}

In this section, the running CEF example is used to illustrate the finite-sample properties of our proposed MTP with FWER level $\alpha=0.05$.
The critical value is computed two ways: first, simulating from the asymptotic normal distribution as in \cref{sec:cv-sim} with $10^5$ draws, and second, using a nonparametric bootstrap as in \cref{sec:cv-bootstrap} with $10^3$ draws.
We also compare with a Bonferroni critical value.
In our setting, FWER is highest under the least favorable null where the function $m(\cdot)$ is flat, so we use a DGP with $h=4$ and $m(1)=m(2)=m(3)=m(4)$.
If the MTP controls FWER in this case, then the FWER will be even lower with strict inequalities like $m(1)<m(2)<m(3)<m(4)$.
Our MTP code is in R \citep{R.core} and uses the multivariate random normal function from the MASS package \citep{R.MASS}.
Results used R version 4.5.1.

Additional simulations are in \cref{sec:power-sim} to illustrate our alternative MTPs that further improve power.

\subsection{Setup}
\label{sec:sim-setup}

We generate the data as follows.
First, we take $n_x$ observations for each value $X_i=x$ ($x=1,2,3,4$), so there are $\sum_{x=1}^{4}n_x=4n_x$ total observations.
Second, we generate the $Y_i$ given its corresponding $X_i=x$.
The conditional distribution is $Y_i\sim\NormDist(1,\sigma_x)$, where $\sigma_x^2\equiv\Var(Y\mid X=x)$ and either $\sigma_x=1$ or $\sigma_x=x$.
Normality is not the most challenging DGP but is sufficient to show patterns as the sample size increases and between our two critical value methods.
The conditional means are all $m(x)=\E(Y\mid X=x)=1$, which maximizes FWER.
That is, if FWER is controlled here, then it would be even lower with different $m(x)$.

Given the DGP, we compute the MTP properties as follows.
First, for each of $\numnornd{100000}$ simulated datasets, we run the MTP in \cref{meth:plain-MTP} with $\alpha=0.05$.
Second, we compute the FWER.
Because all the nulls are true, the FWER is the proportion of simulated datasets in which the MTP rejects at least one $H_{0x}$.
Third, we also report the minimum, median, and maximum critical values $\hat{\hat{c}}_\alpha$, among simulation replications for a particular DGP and method, along with the Bonferroni critical value $\Phi^{-1}(1-\alpha/3)=2.128$.

\subsection{Simulation results}

\Cref{tab:sim-plain} shows the following patterns related to critical values.
First, as $n_x$ increases, the ``Normal'' (\cref{sec:cv-sim}) simulated critical value becomes more ``precise'' in the sense of lower variance across replications because the covariance matrix estimator's variance decreases, but even with $n_x=10$ the range is relatively precise: $[2.120,2.138]$.
The ``Bootstrap'' critical value also becomes more precise, but not as much as the Normal.
Second, also as $n_x$ increases, the median Normal critical value does not change, whereas the median bootstrap critical value dereases.
This reflects a finite-sample advantage of the bootstrap to capture non-normal sampling distributions.
In this case, with $n_x=10$ the marginal $t$-statistic distributions have somewhat thicker-than-normal tails and thus warrant a somewhat larger critical value, which is reflected by the median Bootstrap critical value but not the median Normal.
Third, for this DGP, the Bonferroni critical value $\Phi^{-1}(1-\alpha/3)=2.128$ is nearly asymptotically exact.
The Normal critical values are very close to this, even with the small $n_x=10$, and the corresponding FWER is either the same or within $0.001$.
That is, even in this case that heavily favors Bonferroni, our MTP adapts to get nearly the same results.
(This MTP also introduces the framework used for our refined two-stage procedure in \cref{sec:RSW}.)
Fourth, the MTP critical values are all well above the one-sided $t$-test critical value $1.64$, reflecting the fact that naively running individual $\alpha=0.05$ $t$-tests on each $H_{0x}$ would fail to control the FWER at $\alpha=0.05$.

\begin{table}[htbp]
\centering
\caption{\label{tab:sim-plain}Simulation examples, $\alpha=0.05$.}
\sisetup{round-precision=3}
\begin{tabular}[c]{c
S[table-format=5.0,round-precision=0,round-mode=places]  
l c
S[table-format=1.3,round-precision=3] 
S[table-format=1.3,round-precision=3] 
S[table-format=1.3,round-precision=3] 
S[table-format=1.3,round-precision=3] 
}
\toprule
\multicolumn{1}{c}{$\sigma_x$} & \multicolumn{1}{c}{$n_x$} & 
\multicolumn{1}{c}{Method} & 
\multicolumn{1}{c}{$\hat{\hat{c}}_\alpha$ $[\min,\text{med},\max]$} & 
\multicolumn{1}{c}{\FWER} \\
\midrule
$1$ &     10 &     Normal & [2.120,2.128,2.138] & 0.069500 \\
$1$ &     10 &  Bootstrap & [1.982,2.381,3.511] & 0.039900 \\
$1$ &     10 & Bonferroni &        2.128        & 0.069400 \\
\midrule
$x$ &     10 &     Normal & [2.121,2.128,2.138] & 0.071900 \\
$x$ &     10 &  Bootstrap & [1.938,2.423,4.169] & 0.041200 \\
$x$ &     10 & Bonferroni &        2.128        & 0.071900 \\
\midrule
$1$ &    100 &     Normal & [2.124,2.127,2.135] & 0.052100 \\
$1$ &    100 &  Bootstrap & [1.925,2.146,2.385] & 0.049600 \\
$1$ &    100 & Bonferroni &        2.128        & 0.051900 \\
\midrule
$x$ &    100 &     Normal & [2.124,2.127,2.135] & 0.051000 \\
$x$ &    100 &  Bootstrap & [1.915,2.150,2.385] & 0.049000 \\
$x$ &    100 & Bonferroni &        2.128        & 0.050900 \\
\midrule
$1$ &   1000 &     Normal & [2.125,2.128,2.131] & 0.052700 \\
$1$ &   1000 & Bonferroni &        2.128        & 0.052600 \\
\midrule
$x$ &   1000 &     Normal & [2.125,2.128,2.131] & 0.054500 \\
$x$ &   1000 & Bonferroni &        2.128        & 0.054400 \\
\midrule
$1$ &  10000 &     Normal & [2.126,2.128,2.131] & 0.048400 \\
$1$ &  10000 & Bonferroni &        2.128        & 0.048300 \\
\midrule
$x$ &  10000 &     Normal & [2.126,2.128,2.131] & 0.049200 \\
$x$ &  10000 & Bonferroni &        2.128        & 0.049300 \\
% [1] "2025-08-31 17:15:02 CDT"
% Time difference of 4.689149 hours
% $1$ &     10 &     Normal & [2.120,2.128,2.138] & 0.069500 & 0.025700 & 0.021900 & 0.022100 \\
% $1$ &     10 &  Bootstrap & [1.982,2.381,3.511] & 0.039900 & 0.015300 & 0.012600 & 0.012100 \\
% $x$ &     10 &     Normal & [2.121,2.128,2.138] & 0.071900 & 0.025900 & 0.023700 & 0.022600 \\
% $x$ &     10 &  Bootstrap & [1.938,2.423,4.169] & 0.041200 & 0.016500 & 0.012900 & 0.012000 \\
% $1$ &    100 &     Normal & [2.124,2.127,2.135] & 0.052100 & 0.017900 & 0.018300 & 0.016000 \\
% $1$ &    100 &  Bootstrap & [1.925,2.146,2.385] & 0.049600 & 0.017100 & 0.017700 & 0.014900 \\
% $x$ &    100 &     Normal & [2.124,2.127,2.135] & 0.051000 & 0.016900 & 0.016800 & 0.017400 \\
% $x$ &    100 &  Bootstrap & [1.915,2.150,2.385] & 0.049000 & 0.015900 & 0.016800 & 0.016400 \\
% $1$ &   1000 &     Normal & [2.125,2.128,2.131] & 0.052700 & 0.016400 & 0.018600 & 0.018300 \\
% $x$ &   1000 &     Normal & [2.125,2.128,2.131] & 0.054500 & 0.017900 & 0.017800 & 0.019200 \\
% $1$ &  10000 &     Normal & [2.126,2.128,2.131] & 0.048400 & 0.014800 & 0.015200 & 0.018400 \\
% $x$ &  10000 &     Normal & [2.126,2.128,2.131] & 0.049200 & 0.016400 & 0.014100 & 0.018900 \\
% % Time difference of 4.922982 hours
\bottomrule
\end{tabular}
\end{table}

\Cref{tab:sim-plain} also shows patterns in the error rates.
First, with the Normal critical value, the FWER is near the nominal $\alpha=0.05$ at larger sample sizes, and even with $n_x=10$ the distortion is modest, with $7\%$ FWER.
Second, the Bonferroni error rates are nearly identical to the Normal rates in every case.
Third, compared to both Normal and Bonferroni, the Bootstrap critical value improves FWER at the smallest $n_x=10$, although even by $n_x=100$ this advantage has essentially disappeared.
(This is why we did not spend several more hours to run Bootstrap on the larger $n_x$ simulations.)
Of course, the specific $n_x$ where the advantage disappears depends on the underlying distribution, and more generally depends on the model and estimator, but generally this shows that it may be worth using bootstrap with smaller samples, but not worth the extra computation time on large samples.
As noted above, the FWER improvement seems due to the larger median Bootstrap critical value that captures some of the non-normality of the finite-sample $t$-statistic distributions.

\section{Procedures to Improve Power}
\label{sec:power}

We develop two ways to improve power without sacrificing strong control of FWER.
\Cref{sec:stepdown} describes a stepdown procedure that improves power when there are multiple false null hypotheses and at least one rejection by the original MTP in \cref{meth:plain-MTP}.
Essentially, the rejected hypotheses can be ignored when recomputing a smaller critical value that increases power against the remaining false hypotheses.
\Cref{sec:RSW} describes a two-stage procedure following the strategy of \citet{RomanoShaikhWolf2014}.
Complementing the stepdown, this two-stage procedure improves power when there are null hypotheses that are true and not binding, meaning the true $d_x$ is strictly below zero.
Roughly speaking, the procedure ``estimates'' how far from binding they are (how far below zero are those $d_x$) and lowers the critical value to account for the corresponding $t$-statistics being less likely to cause false rejections.
However, there can be DGPs for which the two-stage power is worse than the plain MTP in \cref{meth:plain-MTP}, like if all true null hypotheses are indeed binding equalities.
In contrast, by construction, the stepdown's power is always at least as high as the plain MTP's power.

\subsection{Stepdown procedure}
\label{sec:stepdown}

The idea of a stepdown procedure is from \citet{Holm1979}; see also Procedure 9.1.1 of \citet[\S9.1]{LehmannRomano2022text}.
Generally, a stepdown procedure starts by deciding whether or not to reject the hypothesis that has the largest test statistic.
If the corresponding null hypothesis is not rejected, then none of the other nulls are rejected, and the procedure stops.
If the null is rejected, then the critical value is adjusted, and the procedure moves to the next-largest test statistic.
Following this logic, \cref{meth:stepdown} describes our stepdown procedure.

\begin{method}[stepdown]
\label{meth:stepdown}
Run the following steps.
\begin{steps}
\item For iteration $i=0$, run \cref{meth:plain-MTP} with $H_{0x} \colon d_x \le 0$.
Set the iteration counter to $i=1$.
\item\label{meth:stepdown-K} Compute the set $\hat{K}^{(i)}\equiv\{x:H_{0x}\textrm{ not yet rejected}\}$, corresponding to null hypotheses not rejected in any iteration before $i$.
\item\label{meth:stepdown-cv} Compute critical value $c_{\alpha}^{\hat{K}^{(i)}}$ as the $(1-\alpha)$-quantile of the asymptotic distribution of $\max_{x \in \hat{K}^{(i)}} \hat{Z}_x$.
\item Reject any additional $H_{0x}$ for which $\hat{t}_x>c_{\alpha}^{\hat{K}^{(i)}}$.
If there are no additional rejections, or if all $H_{0x}$ have now been rejected, then stop.
Otherwise, increment $i$ by one and return to \Cref{meth:stepdown-K}.
\end{steps}
\end{method}

\Cref{prop: just stepdown} establishes the validity of the stepdown procedure in \cref{meth:stepdown}.

\begin{proposition}
\label{prop: just stepdown}
\Cref{meth:stepdown} tests $H_{0x} \colon d_x\le 0$ across $x\in\{1,2,\ldots,h-1\}$ with strong control of asymptotic \FWER\ at level $\alpha$.
\end{proposition}

The upside of the stepdown procedure is that it weakly improves power for any DGP, but it has other limitations.
First, the critical values in \cref{meth:stepdown} are simulated based on the least favorable null with all not-yet-rejected $d_x=0$ in each step, which may be conservative.
Second, if the original MTP does not have at least one rejection, then the stepdown procedure makes no difference.
The method in \cref{sec:RSW} complements the stepdown procedure by addressing both of these limitations.

\subsection{Two-stage procedure} 
\label{sec:RSW}

Our two-stage procedure follows \citet{RomanoShaikhWolf2014}.
First, we construct a $1-\beta$ confidence region (CR) for the true vector of differences $d_x=m(x)-m(x+1)$ using a small $\beta<\alpha$.
Second, we run the MTP at level $\alpha-\beta$ with the critical value calibrated to the least favorable null within the CR.
\Cref{sec:RSW-1,sec:RSW-2} detail these two stages.
Our original MTP in \cref{meth:plain-MTP} can be seen as the special case with $\beta=0$, in which case the CR is the entire parameter space and thus always includes the overall least favorable null with all $d_x=0$.

\subsubsection{First stage: confidence region}
\label{sec:RSW-1}

The first stage is to construct a CR that jointly covers the true differences $d_x$ with high asymptotic probability, and then to find the least favorable null within this CR.
\Cref{prop:CR} establishes the asymptotic validity of the CR in \cref{meth:CR}, which implicitly inverts an MTP like \cref{meth:plain-MTP} but reversing the direction of the $H_{0x}$ inequalities, implicitly similar to the CR in (4) of \citet{RomanoShaikhWolf2014}.

\begin{method}[confidence region]
\label{meth:CR}
The CR for the true vector $\vecf{d}=(d_1,\ldots,d_{h-1})'$ is 
\begin{equation*}
\widehat{\mathrm{CR}} \equiv
\bigl\{ \vecf{\delta} : \delta_x \le \hat{d}_x - c_\beta \hat{s}_x \; \forall x\in\{1,\ldots,h-1\} \bigr\}
,
\end{equation*}
where $c_\beta$ is the $\beta$-quantile of the asymptotic distribution of $\min_{x\in\{1,\ldots,h-1\}} \hat{Z}_x$, $d_x\equiv m(x)-m(x+1)$ is the true difference, $\hat{d}_x=\hat{m}(x)-\hat{m}(x+1)$ is the estimated difference, $\hat{s}_x$ is the estimated asymptotic standard error of $\hat{d}_x$, and the asymptotic joint distribution of the $\hat{Z}_x\equiv(\hat{d}_x-d_x)/\hat{s}_x$ is in \cref{supp_eqn:Zhat-asy-dist}.
\end{method}

\begin{proposition}
\label{prop:CR}
Under \Cref{a:asy-normal}, the CR from \cref{meth:CR} covers the true point $\vecf{d}$ with asymptotic probability $1-\beta$: $\Pr(\vecf{d}\in\widehat{\mathrm{CR}})=1-\beta+o(1)$.
\end{proposition}

The least favorable null within the CR is characterized as follows.
The CR is an orthant whose corner point has components $\hat{d}_x - c_\beta\hat{s}_x$ for $x=1,\ldots,h-1$.
If the CR indeed contains the true $\vecf{d}$, then the least favorable null is the least-negative point inside the intersection of this CR and the null hypothesis region with each $d_x\le0$.
That is, the hypothetical $\vecf{d}$ within the CR that would generate the highest FWER is the point
\begin{equation}
\label{eqn:d-hat-star}
\min(\vecf{0}, \hat{\vecf{d}}^*)
\equiv
\bigl(
\min\{0,\hat{d}_1^*\} ,
\dots,
\min\{0,\hat{d}_{h-1}^*\}
\bigr)'
,\quad
\hat{d}_x^* \equiv \hat{d}_x-c_\beta \hat{s}_x .
\end{equation}

\subsubsection{Second stage: MTP}
\label{sec:RSW-2}

The second stage is to run an MTP like \cref{meth:plain-MTP} but with two modifications.
First, the FWER level is adjusted from $\alpha$ to $\alpha-\beta$ to account for possible error in the first stage.
Second, instead of calibrating the critical value to the overall least favorable null $\vecf{d}=\vecf{0}$, we calibrate the critical value to the new least favorable null in \cref{eqn:d-hat-star}.
\Cref{meth:RSW} describes this modified MTP, whose asymptotic validity is given in \cref{res:RSW}.

\begin{method}
\label{meth:RSW}
After running \cref{meth:CR}, the MTP rejects $H_{0x} \colon d_x \le 0$ when $\hat{t}_x > \hat{c}$, where the $t$-statistics are $\hat{t}_x=\hat{d}_x/\hat{s}_x$ with estimated standard error $\hat{s}_x$ defined below, and critical value $\hat{c}$ is defined as the $(1-\alpha+\beta)$-quantile of the $\max$ over $h-1$ jointly normal random variables with mean vector $\min(\vecf{0}, \hat{\vecf{d}}{}^*)/ \hat{\vecf{s}}$ and covariance matrix $\hat{\matf{\Sigma}}$ defined below, where
\begin{equation*}
\frac{\min(\vecf{0}, \hat{\vecf{d}}{}^*)}{\hat{\vecf{s}}}
\equiv \biggl(\frac{\min\{0, \hat{d}_1^*\}}{\hat{s}_1}, 
\dots, 
\frac{\min\{0, \hat{d}_{h-1}^*\}}{\hat{s}_{h-1}} \biggr) .
\end{equation*}
That is, $\hat{c}$ is the $(1-\alpha+\beta)$-quantile of $\max \{ \NormDist( \min(\vecf{0}, \hat{\vecf{d}}{}^*)/\hat{\vecf{s}},\hat{\matf{\Sigma}}) \}$, with $\hat{s}_x= \sqrt{\hat{V}^b_{(xx)}/n}$, where $\hat{\matf{V}}{}^b=\matf{G}' \hat{\matf{V}}{}^a \matf{G}$ as in \cref{supp_eqn:dhat-asy-dist}, with $\hat{\matf{V}}{}^a$ the consistent estimator of $\matf{V}^a$ from \cref{a:asy-normal}.
From \cref{supp_eqn:Zhat-asy-dist} and the text before and after it that defines $\matf{A}$ and $\matf{\hat{A}}$, $\matf{\Sigma} \equiv \matf{A} \matf{V}^b \matf{A}'=\matf{A} \matf{G}' \matf{V}^a \matf{G} \matf{A}'$, so $\hat{\matf{\Sigma}}= \matf{\hat{A}}\matf{G}'\hat{\matf{V}}{}^a  \matf{G} \matf{\hat{A}}'$.
\end{method}

In practice, $\hat{c}$ must be simulated because there is no closed-form expression for $\hat{c}$.
Similar to \cref{sec:cv-sim}, we take random draws of the Gaussian vector with the appropriate mean and covariance, and then we take the $(1-\alpha+\beta)$-quantile of the maxima of the vectors to get the simulated $\hat{c}$, denoted $\hat{\hat{c}}$.

\begin{theorem}
\label{res:RSW}
\Cref{meth:RSW} has strong control of asymptotic FWER at level $\alpha$.
\end{theorem}

\subsection{Simulation with refined MTPs}
\label{sec:power-sim}

We provide a small simulation to illustrate the potential power improvement of the stepdown and two-stage procedures.
We continue the CEF example but with different $m(x)$ values.
We compare the FWER and power of five different methods: the plain MTP (\cref{meth:plain-MTP}), Bonferroni, stepdown (\cref{meth:stepdown}), two-stage (\cref{meth:RSW}), and combined stepdown and two-stage.
As in \cref{sec:sim}, we used version 4.5.1 of R \citep{R.core} and the MASS package \citep{R.MASS}.

\subsubsection{Setup}
\label{sec:power-sim-setup}

The DGP is designed to have some features advantageous for the stepdown procedure and other features advantageous for the two-stage procedure.
We have $h=23$ values, $x\in\{1,\ldots,23\}$.
The stepdown helps when there are some initial rejections, so a smaller critical value can be used in the next iteration.
Thus, the stepdown benefits from some $H_{0x}$ being clearly violated, which in our context means $m(\cdot)$ decreasing steeply at certain points.
In our DGP, $m(\cdot)$ decreases from $x=12$ to $x=23$, with especially steep decreases from $x=13$ to $x=15$.
Complementing these stepdown benefits, the two-stage helps most when there are some $H_{0x}$ that are clearly satisfied, which means $m(\cdot)$ increasing.
In our DGP, $m(\cdot)$ increases from $x=2$ to $x=12$.
Because the stepdown and two-stage procedures complement each other, the combined power improvement is bigger than the individual improvements, rejecting more of the $16\le x\le 23$ points where $m(\cdot)$ is slowly decreasing and thus $H_{0x}$ is moderately violated.
We also have $m(1)=m(2)=0$, so $H_{01}$ is the most susceptible to false rejections.
Altogether, we set
\begin{equation}
\label{eqn:power-sim-m}
m(x) = \begin{cases}
0 & \text{if }\phantom{1}1\le x\le 2 ,\\[-0pt]
10(x-2) & \text{if }\phantom{1}3\le x\le 12 ,\\[-0pt]
100-10(x-12) & \text{if }13\le x\le 15 ,\\[-0pt]
70-0.42(x-15) & \text{if }16\le x\le 23 .
\end{cases}
\end{equation}

The data generation steps are the same as in \cref{sec:sim-setup}.
First, we take $n_x$ observations for each possible $X_i=x$ value, for $n_x\in\{50,100,200\}$ in turn.
% , so there are $\sum_{x=1}^{23}n_x=2300$ total observations.
Second, we generate the $Y_i$ given its corresponding $X_i=x$.
The conditional distribution is $Y_i\mid X_i\sim\NormDist(m(X_i),1)$, with $m(\cdot)$ from \cref{eqn:power-sim-m}.

We simulate $1000$ datasets and compute the properties of the five methods.
For each simulated dataset, we run the plain MTP (\cref{meth:plain-MTP}), Bonferroni, stepdown (\cref{meth:stepdown}), two-stage (\cref{meth:RSW}), and combined stepdown and two-stage.
Each is run with nominal FWER level $\alpha=0.05$, and with $10^5$ draws to compute the critical value.
The two-stage procedure uses a first-stage $99\%$ CR by setting $\beta=0.01$.
Like before, we compute FWER as the proportion of simulated datasets in which any true null hypothesis is rejected.
Given \cref{eqn:power-sim-m}, $H_{0x}\colon m(x)\le m(x+1)$ is true for $1\le x\le 12$.
We also compute the power.
For each false $H_{0x}$, the simulated pointwise power is the proportion of simulated datasets in which that $H_{0x}$ is rejected.
To summarize power, for each method, we sum its pointwise power across all $13\le x\le 23$.
This can also be interpreted as the average number of false $H_{0x}$ rejected.
In the far right column of \cref{tab:sim2}, this is also expressed as a percentage of the total $11$ false $H_{0x}$.
For example, if on average $5.5$ false $H_{0x}$ are rejected, then this translates to $(5.5/11)\times100\%=50\%$.

\subsubsection{Simulation results}

\begin{table}[htb]
\centering
\caption{\label{tab:sim2}Simulation results, $\alpha=0.05$, $\beta=0.01$.}
\sisetup{round-precision=3}
\begin{tabular}[c]{
rl
S[table-format=1.3,round-precision=3]
S[table-format=2.2,round-precision=2]
S[table-format=2.1,round-precision=1]
}
\toprule
$n_x$ & \multicolumn{1}{c}{Method} & \multicolumn{1}{c}{\FWER} & 
\multicolumn{1}{c}{Power (sum)} & \multicolumn{1}{c}{Power (\% of $11$)} \\
\midrule
   50 & Bonferroni &  0.00400 &  4.96800 &  45.1636 \\
   50 &  Plain MTP &  0.00400 &  4.99800 &  45.4364 \\
   50 &   Stepdown &  0.00500 &  5.20100 &  47.2818 \\
   50 &  Two-stage &  0.00500 &  5.31800 &  48.3455 \\
   50 &   Combined &  0.01100 &  5.84900 &  53.1727 \\
% [1] 2.837597
%       0%      25%      50%      75%     100% 
% 2.818814 2.829072 2.831758 2.834468 2.844583 
\midrule
  100 & Bonferroni &  0.00400 &  7.46100 &  67.8273 \\
  100 &  Plain MTP &  0.00400 &  7.47100 &  67.9182 \\
  100 &   Stepdown &  0.00600 &  7.88100 &  71.6455 \\
  100 &  Two-stage &  0.00600 &  7.86200 &  71.4727 \\
  100 &   Combined &  0.02400 &  9.11200 &  82.8364 \\
% [1] 2.837597
%       0%      25%      50%      75%     100% 
% 2.820491 2.829258 2.832047 2.835068 2.843933 
\midrule
  200 & Bonferroni &  0.00100 & 10.31700 &  93.7909 \\
  200 &  Plain MTP &  0.00100 & 10.32600 &  93.8727 \\
  200 &   Stepdown &  0.00100 & 10.53400 &  95.7636 \\
  200 &  Two-stage &  0.00100 & 10.46600 &  95.1455 \\
  200 &   Combined &  0.02600 & 10.87600 &  98.8727 \\
% [1] 2.837597
%       0%      25%      50%      75%      100%
% 2.820038 2.829565 2.832495   2.835024  2.844827 
\bottomrule
\end{tabular}
% [1] "2025-08-27 13:24:36 CDT"
% Time difference of 2.618274 hours
% 
% \begin{threeparttable}
% \begin{tabular}[c]{
% l
% S[table-format=1.3,round-precision=3]
% S[table-format=1.2,round-precision=2]
% }
% \toprule
% \multicolumn{1}{c}{Method} & \multicolumn{1}{c}{\FWER} & \multicolumn{1}{c}{Power (sum)} \\
% \midrule
% Bonferroni &  0.00400 &  7.46100 \\
%  Plain MTP &  0.00400 &  7.47100 \\
%   Stepdown &  0.00600 &  7.88100 \\
%  Two-stage &  0.00600 &  7.86200 \\
%   Combined &  0.02400 &  9.11200 \\
% \bottomrule
% \end{tabular}
% % \end{threeparttable}
% % [1] 2.837597
% %       0%      25%      50%      75%     100% 
% % 2.820491 2.829258 2.832047 2.835068 2.843933 
\end{table}

\Cref{tab:sim2} shows the results, with the following patterns.
First, all five methods control FWER below the desired $\alpha=0.05$ level.
The plain MTP and Bonferroni have the smallest FWER, and the combined method the largest, but still well below $\alpha$.
Unlike in \cref{tab:sim-plain}, where the DGP was the least favorable null that leads to asymptotically exact FWER, the DGP here is far from least favorable in order to study power rather than FWER control, so the FWER will not approach $\alpha$ even with larger samples.
Second, the power improvements can be seen.
The qualitative patterns are the same for each $n_x$, so we focus on the $n_x=100$ results in more detail.
Individually, the stepdown and two-stage procedures both improve the aggregate power by around $0.4$; that is, compared to the plain MTP rejecting $7.5$ false $H_{0x}$ on average, they each reject $7.9$.
The combined improvement is even larger because the stepdown further reduces the critical value after the additional rejections from the two-stage method.
However, recall that the DGP was designed to have features that showcase the stepdown and two-stage procedures, so the improvements will not always be so large.
Finally, we can see that as $n_x$ increases, power increases toward $100\%$.

\section{Conclusion}

We have provided new multiple testing methods and corresponding confidence sets to provide richer results about a function's monotonicity than testing a single global null hypothesis.
Our assumptions cover a wide range of descriptive and causal statistical models.
This work can extend in several directions, including an asymptotically increasing number of evaluation points that leverages \citet{ChernozhukovEtAl2019many}, and the development of Bayesian credible sets.

For continuous $X$, another extension is to combine an existing global monotonicity test with the closure method \citep[e.g.,][\S9.2]{LehmannRomano2022text} as follows.
First, partition the support $\mathcal{X}$ into intervals $\mathcal{X}_j$ for $j=1,\ldots,h$.
Second, run the global test on every possible combination of such intervals.
Third, reject that the function is increasing over interval $\mathcal{X}_j$ if increasingness is rejected at level $\alpha$ for every combination of intervals that includes $\mathcal{X}_j$.
As shown by \citet[\S9.2.1]{LehmannRomano2022text}, this controls the familywise error rate at level $\alpha$.

\backmatter

% \singlespacing

\bmhead{Supplementary information}
% \section*{Supplementary Materials}

% The supplementary appendix has extensions to IVQR and functional coefficient models, full proofs of all our theoretical results, and an additional empirical example.
The supplementary appendix has extensions to IVQR and functional coefficient models as well as full proofs of all our theoretical results.
Also provided is R code implementing our new methods and replicating all simulation and empirical results.

% [...but published paper has it here!] TEST said to put on title page
\bmhead{Acknowledgments}

Thanks to the following for their helpful feedback: Alyssa Carlson, Zack Miller, Shawn Ni, journal reviewers and editors, and participants in the annual meetings of the Missouri Valley Economic Association (2021) and Midwest Econometrics Group (2023).
This work is based on a chapter of the first author's PhD dissertation.

\section*{Declarations}

\bmhead{Conflict of interest}
We (the authors) declare no conflict of interest.

\bmhead{Data availability}
All data is available publicly and loaded automatically through the provided R code.

\bmhead{Code availability}
All code is available on the second author's website.%
\footnote{\url{https://kaplandm.github.io/}}

% \begin{itemize}
% \item Funding: none
% \item Conflict of interest: we (the authors) declare no conflict of interest.
% \item Ethics approval and consent to participate: n/a
% \item Consent for publication: yes
% \item Data availability: yes
% \item Materials availability: n/a
% \item Code availability: yes
% \item Author contribution: n/a
% \end{itemize}

% \clearpage


 %\bibliography{_bib}

\clearpage

%%%%%APPENDIX %%%%%%

\begin{appendices}

% \noindent
\begin{center}
\LARGE\singlespacing
Supplementary Appendix for

``Multiple Testing of a Function's Monotonicity'' 

\large ~\\[-6pt] by Wei Zhao and David M.\ Kaplan

\large August 31, 2025 %\today
\end{center}

\setcounter{page}{1}

\section{IVQR and Functional Coefficient Model Examples}
\label{sec:app-ext}

\Cref{sec:IVQR,sec:functional-coefficient} show how our methodology can apply to instrumental variables quantile regression (IVQR) and functional coefficient models, respectively, with some small modifications.

\subsection{IVQR}
\label{sec:IVQR}

Consider a causal random coefficient model
\begin{equation*}
Y = \vecf{X}'\vecf{\beta}(U), \quad U \sim \UnifDist(0,1) ,
\end{equation*}
with regressor vector and random coefficient vector respectively
\begin{equation*}
\begin{split}
\vecf{X} &= (1,X_1,\ldots,X_p)' , \\
\vecf{\beta}(U) &= (\beta_0(U), \beta_1(U), \ldots, \beta_p(U)) .
\end{split}
\end{equation*}
Each component function $\beta_j(\cdot)$ is non-random, and $U$ is an unobserved scalar random variable.
The endogeneity is in terms of the slope coefficients' correlation with the regressors.
This is a special case of the IVQR model proposed by \citet{ChernozhukovHansen2005}.

Imagine we want to learn which values of $U$ correspond to an increasing relationship between $Y$ and $X_1$.
Mathematically, this is equivalent to asking at which $u$ values is $\beta_1 (u)> 0$.
Choosing a grid of $u$ values, the null hypotheses of interest are
\begin{equation}
\label{eqn:IVQR-nulls}
H_{0j}\colon \beta_1(u_j) \le 0
,\quad
j=1,\ldots,h .
\end{equation}

\Cref{a:IVQR} is essentially a special case of \Cref{a:asy-normal}, replacing $m$ with $\beta_1$, and $x$ with $u$.

\begin{assumpIVQR}
\label{a:IVQR}
The estimator $\hat{\vecf{\beta}}_1\equiv(\hat{\beta}_1(u_1),\hat{\beta}_1(u_2), \ldots,\hat{\beta}_1(u_h))'$ is asymptotically normal given fixed $0<u_1<\cdots<u_h<1$: given true value $\vecf{\beta}_1\equiv(\beta_1(u_1),\beta_1(u_2), \ldots, \beta_1(u_h))'$,
\begin{equation*}
\sqrt{n}(\hat{\vecf{\beta}}_1-\vecf{\beta}_1)
\dconv \NormDist(\vecf{0},\matf{\Omega}) ,
\end{equation*}
where positive definite matrix $\matf{\Omega}$ can be estimated consistently, $\hat{\matf{\Omega}}\pconv\matf{\Omega}$. 
\end{assumpIVQR}

\Cref{a:IVQR} is justified by results from \citet{ChernozhukovHansen2006}.
They derive the limiting Gaussian process of $\sqrt{n}(\hat{\vecf{\beta}}(\cdot)-\vecf{\beta}(\cdot))$.
Specifically, their Theorem 3 directly implies \Cref{a:IVQR} because their result is joint over components of $\vecf{\beta}$ as well as a continuum of $u$ values, and \Cref{a:IVQR} is taking marginal distributions of specific $\beta_1(u)$ over only a finite number of $u$ values.
Also, in Remark 4, they provide a way to estimate the components of the asymptotic variance.

\Cref{meth:IVQR-MTP} describes the MTP.
It is very similar to \cref{meth:plain-MTP}, but replacing $\hat{d}$ with $\hat\beta_1$, and $x$ with $u$.
Because we are now interested in $\beta_1(u_j)$ itself rather than a difference like $m(x_j)-m(x_{j+1})$, the joint asymptotic distribution of the $\hat{Z}_j$ differs, but the general approach is the same.
Below the method, \cref{res:IVQR} states its validity.

\begin{method}
\label{meth:IVQR-MTP}
Given the $H_{0j}$ in \cref{eqn:IVQR-nulls}, the MTP rejects $H_{0j}$ when $\hat{t}_j>c_\alpha$, where the $t$-statistics are $\hat{t}_j=\hat{\beta}_1(u_j)/\hat{s}_j$, the estimated standard errors are $\hat{s}_j\equiv \sqrt{\hat{\Omega}_{jj} / n}$ given the consistent estimator $\hat{\matf{\Omega}}$ from \Cref{a:IVQR}, and $c_\alpha$ is the critical value.
Specifically, $c_\alpha$ is the $(1-\alpha)$-quantile of the asymptotic distribution of $\hat{Z}^*$, which is the maximum of $h$ correlated normal random variables $\hat{Z}_j$:
\begin{equation*}
\hat{Z}^*\equiv\max_{j\in\{1,2,\dots,h\}} \hat{Z}_j
,\quad
\hat{Z}_j \equiv \frac{\hat{\beta}_1(u_j)-\beta_1(u_j)}{\hat{s}_j} \textrm{ for }j\in\{1,2,\dots,h\}
,
\end{equation*}
and the asymptotic joint normal distribution of the $\hat{Z}_j$ is in \cref{eqn:Zhat-asy-dist-IVQR}. 
The details of simulating $c_\alpha$ are as in \cref{sec:cv-sim}, except the $\hat{\matf{\Sigma}}$ is from \cref{eqn:Zhat-asy-dist-IVQR} instead of \cref{eqn:Zhat-asy-dist}.
\end{method}

\begin{theorem}
\label{res:IVQR}
Under \Cref{a:IVQR}, \cref{meth:IVQR-MTP} has strong control of asymptotic FWER.
\end{theorem}

\subsection{Functional coefficient model}
\label{sec:functional-coefficient}

The setup of the functional coefficient model is similar to that of IVQR in \cref{sec:IVQR}.

Consider a functional (varying) coefficient model
\begin{align}
\label{eqn:fc-model}
Y=\vecf{X}'\vecf{\beta}(\vecf{W})+\epsilon ,
\end{align}
where $\vecf{X}=(1, X_1, \ldots, X_p)'$, $\vecf{\beta}(\vecf{W})=(\beta_0(\vecf{W}), \beta_1(\vecf{W}),\ldots, \beta_p(\vecf{W}))'$ is the coefficient vector, each component $\beta_j(\cdot)$ is a non-random function, and $\vecf{W}$ is a vector of variables that determine the coefficients through the $ \vecf{\beta}(\cdot)$ function.

\Cref{eqn:fc-model} is a general setup.
The simple model with scalar $W$ was proposed by \citet{ClevelandEtAl1992}.
\Citet{HastieTibshirani1993} study the model $Y = X_1\beta_1(W_1) + X_2\beta_2(W_2) + \cdots + X_p\beta_p(W_p) + \epsilon$.
Functional coefficient models have been developed for longitudinal data \citep{SenturkMuller2010}, time series data \citep{CaiEtAl2009}, and other settings.

Similar to \cref{sec:IVQR}, we want to learn which values of $\vecf{W}$ correspond to an increasing relationship between $Y$ and $X_1$.
Mathematically, this is equivalent to asking at which $\vecf{w}$ values is $\beta_1(\vecf{w})>0$.
Choosing a grid of $\vecf{w}$ values, the null hypotheses of interest are
\begin{equation*}
H_{0j}\colon \beta_1(\vecf{w}_j) \le 0
,\quad
j=1,\ldots,h .
\end{equation*} 
Qualitatively, these nulls are the same as the IVQR nulls in \cref{eqn:IVQR-nulls}, just replacing $u_j$ with $\vecf{w}_j$.
Therefore, if the estimators of vector $\vecf{\beta}_1\equiv(\beta_1(\vecf{w}_1), \ldots, \beta_1(\vecf{w}_h))'$ satisfy the joint asymptotic normality of \Cref{a:IVQR}, then the MTP in \cref{meth:IVQR-MTP} has strong control of asymptotic FWER.
For example, for the local linear regression from \citet{ZhuEtAl2012}, they set up a model in a very general way that includes \cref{eqn:fc-model} as a special case; their Theorem 1 implies asymptotic joint normality under centain conditions like proper bandwidth choice.

\section{Proofs}
\label{sec:app-proofs}

\subsection{Proof of \texorpdfstring{\cref{res:asy-Zhat}}{Lemma \ref{res:asy-Zhat}}}

\begin{proof}
We first derive the asymptotic distribution of $\sqrt{n} (\hat{\vecf{d}}-\vecf{d})$, defining $\hat{\vecf{d}}\equiv(\hat{d}_1, \dots, \hat{d}_{h-1})'$ and $\vecf{d}\equiv(d_1, \dots, d_{h-1})'$, so $\hat{\vecf{d}}-\vecf{d}=(\hat{d}_1-d_1, \dots, \hat{d}_{h-1}-d_{h-1}) $.
By the delta method,
\begin{equation}
\label{eqn:dhat-asy-dist}
\sqrt{n} (\hat{\vecf{d}}- \vecf{d}) \dconv \NormDist(\vecf{0},\matf{V}^b) ,
\end{equation}
where $\matf{V}^b=\matf{G}' \matf{V}^a \matf{G}$, and $\matf{G}$ is the $h\times(h-1)$ first derivative (Jacobian) of $\vecf{d}'$ with respect to $\vecf{m}$, with row-$x$ column-$j$ element $G_{(xj)}=\Ind{x=j} - \Ind{x=j+1}$, for $x=1,\ldots,h$ and $j=1,\ldots,h-1$.
Under \Cref{a:asy-normal}, $\hat{\matf{V}}{}^b=\matf{G}'\hat{\matf{V}}{}^a  \matf{G}  \pconv \matf{V}^b$.

Recall from \cref{meth:plain-MTP} that
$\hat{Z}_x\equiv (\hat{d}_x-d_x)/\hat{s}_x$, where $\hat{s}_x$ is the asymptotic standard error of $\hat{d}_x$.
Gathering these,
\begin{equation*}
\begin{aligned}
\hat{\vecf{Z}}\equiv  (
\hat{Z}_1, \dots, \hat{Z}_{h-1}
)
=
\matf{\hat{A}}\sqrt{n}(\hat{\vecf{d}}- \vecf{d})
\end{aligned}
\end{equation*}
where $\matf{\hat{A}}$ is a diagonal matrix with elements $\hat{A}_{(xx)}=1/(\sqrt{n}\hat{s}_x)$, $x=1,\ldots,h-1$.
Thus,
\begin{equation}
\label{eqn:Zhat-asy-dist}
\hat{\vecf{Z}} \dconv \NormDist (\vecf{0},\matf{\Sigma})
,\quad
\matf{\Sigma} \equiv \matf{A} \matf{V}^b \matf{A}'=\matf{A} \matf{G}' \matf{V}^a \matf{G} \matf{A}' ,
\end{equation}
where $\matf{A}$ is a diagonal matrix with generic element $A_{(xx)}=1/\sqrt{V^b_{(xx)}}$, the reciprocal of the square root of the corresponding diagonal element of $\matf{V}^b$.
Under \Cref{a:asy-normal}, $\hat{\matf{\Sigma}}= \matf{\hat{A}}\matf{G}'\hat{\matf{V}}{}^a  \matf{G} \matf{\hat{A}}' \pconv \matf{\Sigma}$.

Therefore, the asymptotic distribution of $\hat{Z}^*$ is a distribution of the maximum of $h-1$ correlated normal random variables with mean vector $\vecf{0}$ and covariance matrix $\matf{\Sigma}$. 
\end{proof}

\subsection{Proof of \texorpdfstring{\cref{res:plain-MTP}}{Theorem \ref{res:plain-MTP}}}
\label{sec:pf-MTP}

\begin{proof}
To establish strong control of \FWER,  the critical value $c_\alpha$ needs to be the $(1-\alpha)$-quantile of the asymptotic distribution of $\hat{Z}^*$.
The following steps (discussed below) show the reason:
\begin{equation}
\label{eqn:strong control of FWER 2}
\begin{aligned}
\FWER &\equiv\Pr(\text{reject any true }H_{0x})\\
 &=\Pr \left (\max_{x\colon H_{0x} \text{ true}}\hat{t}_x > c_\alpha \right)\\
 &\le \Pr \left(\max_{x\colon H_{0x }\text{ true}} \hat{Z}_x > c_\alpha \right)\\
 &\le \Pr \Bigl( \overbrace{\max_{x\in\{1,\ldots,h-1\}} \hat{Z}_x}^{\equiv\hat{Z}^*}> c_\alpha \Bigr) \\
 &\to \alpha   .
\end{aligned}
\end{equation} 

The first equality (definition) is the definition of FWER in \cref{eqn:FWER}.

The second equality holds because the MTP rejects $H_{0x}$ when $\hat{t}_x>c_\alpha$, so the event $\{\text{reject any true }H_{0x}\}$ is the same as the event $\{\text{max}_{\{x\colon H_{0x}\text{ true}\}}\hat{t}_x > c_\alpha\}$.

The first inequality holds because if $H_{0x}$ is true, then $d_x \le 0$, so subtracting the true negative value (instead of zero) makes $\hat{Z}_x\ge\hat{t}_x$:
\begin{equation}
\label{eqn:txhat-le-Zxhat}
\hat{t}_x\equiv \frac{\hat{d}_x}{\hat{s}_x}
\le \frac{\hat{d}_x- d_x}{\hat{s}_x}
\equiv \hat{Z}_x.
\end{equation}

The second inequality is from the expansion of the support of the maximum.
Because $\{x:H_{0x}\textrm{ true}\} \subseteq \{1,\ldots,h-1\}$, taking a maximum over $\{1,\ldots,h-1\}$ yields a weakly larger value than taking a maximum over $\{x:H_{0x}\textrm{ true}\}$.

The final limit holds by definition of $c_\alpha$.
This is the reason why the MTP requires that $c_\alpha$ is $(1-\alpha)$-quantile of the asymptotic distribution of $\hat{Z}^*$.

Note that both inequalities are binding (equalities) in the special case that $d_x=0$ for all $x$.
This is the least favorable null.
Then, asymptotically, $\FWER\to\alpha$ exactly.
In other cases, asymptotic FWER is strictly below $\alpha$.
\end{proof}

\subsection{Proof of \texorpdfstring{\cref{res:outer-CS}}{Corollary \ref{res:outer-CS}}}

\begin{proof}
The equations below turn the coverage probability statement into probability in terms of rejecting null hypotheses.
Then, it is easy to link back to the fact that $\FWER\le \alpha + o(1)$, which is the source of the last inequality here:
\begin{align*}
\Pr(\hat{\mathcal{S}}_o\supseteq\mathcal{S})&=\Pr(\text{no true }H_{0x}\text{ is rejected})\\
 &= \Pr(\text{no }H_{0x}\text{ is rejected with }x\in\mathcal{S} )\\
 &= 1-\Pr(\text{any }H_{0x}\text{ is rejected with }x\in\mathcal{S})\\
 &= 1- \FWER \ge 1-\alpha+o(1) .
\qedhere
\end{align*} 
\end{proof}

\subsection{Proof of \texorpdfstring{\cref{res:inner-CS}}{Corollary \ref{res:inner-CS}}}

\begin{proof}
The only way we could incorrectly include in the inner CS $x$ where $m(x)$ is not increasing is if we falsely reject the true null $H_{0x}^*$.
Again, it is easy to link back to the fact that the probability of even just one false rejection is controlled below $\alpha$, i.e., $\FWER \le \alpha + o(1)$, which is the source of the last inequality here:
\begin{align*}
\Pr(\hat{\mathcal{S}}_i\subseteq\mathcal{S}) &= 1-\Pr(\text{falsely reject any true }H_{0x}^*\colon x\notin\mathcal{S})\\
 &= 1 - \FWER
\ge 1 - \alpha + o(1) .
\qedhere
\end{align*}
\end{proof}

\subsection{Proof of \texorpdfstring{\cref{prop: just stepdown}}{Proposition \ref{prop: just stepdown}}}

\begin{proof}
First, consider the infeasible oracle critical value from \Cref{meth:stepdown-cv} when the true set of the hypotheses $\hat{K}^{(i)}=\mathcal{S}\equiv\{x : H_{0x}$ is true$\}$ is used, i.e., the critical value $c_{\alpha}^\mathcal{S}$, which is the $(1-\alpha)$-quantile of $\hat{Z}^{*\mathcal{S}}$, using notation $\hat{Z}^{*\mathcal{A}}\equiv\max_{x\in\mathcal{A}}\hat{Z}_x$ for any set $\mathcal{A}\subseteq\mathcal{X}$.
Hypothetically, if this oracle critical value were used, then
\begin{align}
\label{eq: oracle ineq}
\lim _{n \rightarrow \infty} \Pr\left(\max _{x \in \mathcal{S}} \hat{t}_x>c_{\alpha}^\mathcal{S}\right)  \overbrace{\le \lim _{n \rightarrow \infty} \Pr\left(\max _{x \in \mathcal{S}} \hat{Z}_x>c_{\alpha}^{\mathcal{S}}\right)}^{\text{since } \hat{t}_x \le \hat{Z}_x \ \text{for all } x \in \mathcal{S}}\underbrace{=\lim _{n \rightarrow \infty} \Pr\left(\hat{Z}^{* \mathcal{S}}>c_{\alpha}^{\mathcal{S}}\right)}_{\text{since } \max _{x \in \mathcal{S}} \hat{Z}_x=\hat{Z}^{* \mathcal{S}}}\overbrace{=\alpha}^{\text{definition of } c_{\alpha}^{\mathcal{S}}}.
\end{align}

Second, the crtical values have a monotonicity property.
Specifically,
\begin{equation*}
\{1,2, \dots, h-1\}=\hat{K}^{(0)}\supseteq \hat{K}^{(1)} \supseteq \cdots
\end{equation*}
because additional $H_{0x}$ can be rejected in every iteration, but once rejected they can never be un-rejected.
Combined with using the same way to get critical values, this implies the critical values are also monotonic over iterations:
\begin{equation*}
c_{\alpha}^{\hat{K}^{(0)}} \ge c_{\alpha}^{\hat{K}^{(1)}} \ge \cdots.
\end{equation*}
Specifically, recalling that $c_{\alpha}^{\hat{K}^{(i)}}$ is the $(1-\alpha)$-quantile of the asymptotic distribution of $\hat{Z}^{* \hat{K}^{(i)}}\equiv \max _{x \in \hat{K}^{(i)}} \hat{Z}_x$, since $\hat{K}^{(0)} \supseteq \hat{K}^{(1)} \supseteq \hat{K}^{(2)} \supseteq \cdots$, for any realization of the underlying $\hat{Z}_x$ then the corresponding $\hat{Z}^*$
realizations are ordered by
\begin{equation*}
\hat{Z}^{*\hat{K}^{(0)}} \ge \hat{Z}^{* \hat{K}^{(1)}} \ge  \cdots,
\end{equation*}
so
\begin{equation*}
(1-\alpha)\textrm{-quantile of }\hat{Z}^{*\hat{K}^{(0)}}
\ge (1-\alpha)\textrm{-quantile of }\hat{Z}^{*\hat{K}^{(1)}}
\ge \cdots,
\end{equation*}
that is,
\begin{equation*}
c_{\alpha}^{\hat{K}^{(0)}} \ge c_{\alpha}^{\hat{K}^{(1)}} \ge \cdots . 
\end{equation*}

According to Theorem 9.1.3 of \citet[\S9.1]{LehmannRomano2022text}, $\FWER \le \alpha$ because \cref{meth:stepdown} is based on critical values $c_{\alpha}^{\hat{K}^{(i)}}$ satisfying their monotonicity requirement and \cref{eq: oracle ineq} holds.
\end{proof}

\subsection{Proof of \texorpdfstring{\cref{prop:CR}}{Proposition \ref{prop:CR}}}

\begin{proof}
The true point $\vecf{d}=(d_1, \ldots, d_{h-1})$ is included in the CR if and only if $d_x \le \hat{d}_x - c_\beta\hat{s}_x$ for all $x \in\{1,\ldots, h-1\}$, which is equivalent to $\min_{x\in\{1,\ldots, h-1\}} \hat{Z}_x \ge c_\beta$.
Mathematically, 
\begin{align*}
\Pr(\text{CR covers the true point})
  &= \Pr(d_x \le \hat{d}_x - c_\beta\hat{s}_x \text{ for all } x \in\{1,\ldots, h-1\} )
\\&= \notag
\Pr(\min_{x\in\{1,\ldots, h-1\}} \hat{Z}_x \ge c_\beta ) \underbrace{\to 1-\beta}_{\text{by definition of } c_\beta}.
\qedhere
\end{align*}
\end{proof}

\subsection{Proof of \texorpdfstring{\cref{res:RSW}}{Theorem \ref{res:RSW}}}

\begin{proof}
Overall,
\begin{align}\notag
\FWER
  &\equiv \Pr(\text{reject any true } H_{0 x})
   = \Pr\bigl(\max _{x \in \mathcal{S}} \hat{t}_x>\hat{c}\bigr) \\ \notag
  &=\Pr\bigl[\bigl(\max _{x \in \mathcal{S}} \hat{t}_x>\hat{c} \text{ and } \hat{c} \ge \hat{c}^o\bigr) \text{ or }\bigl(\max _{x \in \mathcal{S}} \hat{t}_x>\hat{c} \text{ and } \hat{c}<\hat{c}^o\bigr)\bigr] \\
& \le \underbrace{\Pr\bigl(\max_{x \in \mathcal{S}} \hat{t}_x>\hat{c} \text{ and } \hat{c} \ge \hat{c}^o\bigr)}_{\text{left part}}+ \underbrace{\Pr\bigl(\max_{x \in \mathcal{S}} \hat{t}_x>\hat{c} \text{ and } \hat{c}<\hat{c}^o\bigr)}_{\text{right part}} .
\label{eqn:left-right}
\end{align}
The oracle critical value $\hat{c}^o$ is the $(1-\alpha+\beta)$-quantile of the $\max$ of the Gaussian random vector with mean vector $\vecf{d}/ \hat{\vecf{s}}$ and covariance matrix $\hat{\matf{\Sigma}}$ over $x\in \mathcal{S}$, i.e., the $(1-\alpha+\beta)$-quantile of $\max_{x \in \mathcal{S}}  \{ \NormDist ( \vecf{d}/ \hat{\vecf{s}},\hat{\matf{\Sigma}}) \}$.
Superscript $o$ is for ``oracle'' because it relies on the true $\vecf{d}$, while the ``hat'' is because it uses the estimated $\hat{\matf{\Sigma}}$ and $\hat{\vecf{s}}$.

To bound the ``right part'' in \cref{eqn:left-right},
\begin{align}\label{eqn:c-hat-o-inequality}
& \Pr\bigl(\max _{x \in \mathcal{S}} \hat{t}_x>\hat{c} \text { and } \hat{c}<\hat{c}^o\bigr)
 \le \Pr(\hat{c} < \hat{c}^o)\\
&\le \underbrace{\Pr(\text{CR from \cref{meth:CR} fails to cover the true point})
 \to \beta}_{\textrm{by \cref{prop:CR}}} .
\label{eqn:c-hat-CR-inequality}
\end{align}

The inequality in \cref{eqn:c-hat-o-inequality} holds because $\Pr(A \cap B) \le \Pr(B)$ (because $A\cap B\subseteq B$).

The inequality in \cref{eqn:c-hat-CR-inequality} holds because ``CR from \cref{meth:CR} fails to cover the true point'' is a necessary (but not sufficient) condition for $\hat{c} < \hat{c}^o$.
To prove that it is a necessary condition, we show that $\hat{c} < \hat{c}^o$ implies ``CR fails to cover.''
This can be done by equivalently showing that the contrapositive holds, i.e., that ``CR covers the true point'' implies $\hat{c} \ge \hat{c}^o$:
\begin{align*}%\notag
\hat{c}
  &\equiv (1-\alpha+\beta)\text{-quantile of }\max_{x \in\{1, \dots, h-1\}}  \{ \NormDist ( \min(\vecf{0}, \hat{\vecf{d}}{}^*)/\hat{\vecf{s}}, \hat{\matf{\Sigma}}) \} \\ %\notag
  & \ge (1-\alpha+\beta)\text{-quantile of } \max_{x \in \mathcal{S}}  \{ \NormDist ( \min(\vecf{0}, \hat{\vecf{d}}{}^*)/\hat{\vecf{s}}, \hat{\matf{\Sigma}}) \}  \\ %\notag 
  &\underbrace{\ge (1-\alpha+\beta)\text{-quantile of } \max_{x \in \mathcal{S}}  \{ \NormDist(\vecf{d}/\hat{\vecf{s}} , \hat{\matf{\Sigma}}) \}}_{\text{\llap{because }}\min\{0, \hat{d}_x^*\} \ge d_x \text{ for }x\in \mathcal{S}\text{ when CR covers the true point}}
% \\&
\equiv \hat{c}^o .
\end{align*}

To bound the ``left part'' in \cref{eqn:left-right},
\begin{align}
\label{eqn:left-c-hat-inequality}
\Pr(\max _{x \in \mathcal{S}} \hat{t}_x>\hat{c} \text{ and } \hat{c} \ge \hat{c}^o)
  &\le \Pr(\max _{x \in \mathcal{S}} \hat{t}_x \ge \hat{c}^o)
\\&\to \Pr(t^{\text{max}}\ge c^o )
   \label{eqn:left-t-max-limit}
\\&= \alpha - \beta ,
     \label{eqn:left-final}
\end{align}
where $c^0$ is defined below, and $t^\textrm{max}$ is defined below in \cref{eqn:tmax}.

The inequality in \cref{eqn:left-c-hat-inequality} holds because $\Pr(A\ge B \ge C) \le P(A\ge C)$, because $A\ge B \ge C \implies A\ge C$.

The convergence in \cref{eqn:left-t-max-limit} is determined by three factors.
First, $\hat{c}^o \pconv c^o$, where $c^o$ is defined as the the probability limit of $\hat{c}^o$.
Second, $\max_{x \in \mathcal{S}} \hat{t}_x \dconv t^{\text{max}}$, where $t^{\text{max}}$ has a non-degenerate distribution.
Third, \cref{lemma:convergence} allows the first two points to be combined into the limit shown.
These three points are detailed below.

First, $c^o$ is defined as the the probability limit of $\hat{c}^o$, $c^o \equiv \plimn \hat{c}^o$.
As defined above after \cref{eqn:left-right}, $\hat{c}^o$ is the $(1-\alpha+\beta)$-quantile of $ \max_{x \in \mathcal{S}}  \{ \NormDist ( \vecf{d}/\hat{\vecf{s}},\hat{\matf{\Sigma}}) \}$.
Hypothetically, if we modeled $d_x<0$ as asymptotically fixed, then the mean vector element $d_x/\hat{s}_x \pconv -\infty$ as $n\to\infty$ because $\hat{s}_x= \sqrt{\hat{V}^b_{(xx)} / n} \pconv 0$, because $\hat{V}^b_{(xx)} \pconv V^b_{(xx)}<\infty$.
Hence, if all $d_x<0$, then $c^o \equiv \plim \hat{c}^o= -\infty$.
This is not a helpful asymptotic approximation, so we use a more accurate approximation below.
With fixed $d_x$, only when $d_x=0$ does $d_x/\hat{s}_x$ not diverge, so $c^o \equiv \plim \hat{c}^o= (1-\alpha+\beta)$-quantile of $\max_{x: d_x=0} \{ \NormDist (\vecf{0},\matf{\Sigma}) \}$.

To get a more accurate approximation, let $d_x$ be a sequence (with subscript $n$ left implicit),
\begin{equation*}
d_x \equiv \gamma_x/\sqrt{n}, 
\end{equation*}
where $\gamma_x$ is a constant for each $x$.
Recalling the definitions of $\hat{t}_x$ from \cref{meth:RSW} and of $\hat{Z}_x$ from \cref{meth:plain-MTP},
\begin{align*}%\notag
\hat{t}_x
  &= \hat{Z}_x +\frac{d_x}{\hat{s}_x}
   = \hat{Z}_x +\underbrace{\frac{\gamma_x}{ \sqrt{n} \hat{s}_x}}_{d_x=\gamma_x/\sqrt{n}}
,\\
\sqrt{n} \hat{s}_x
  &= \sqrt{n} \sqrt{\frac{\hat{V}^b_{(xx)}}{n}}
   = \sqrt{\hat{V}^b_{(xx)}} \pconv \sqrt{V^b_{(xx)}} ,
\end{align*}
thus
\begin{align}
\label{eqn:tx-vs-Zx}
\hat{t}_x
  &= \hat{Z}_x + \frac{\gamma_x}{\sqrt{\hat{V}^b_{(xx)}}}
,\\
\label{eqn:t-asy-dist}
\hat{\vecf{t}}
  &\equiv (\hat{t}_1, \dots, \hat{t}_{h-1})
   = (\hat{Z}_1+\frac{\gamma_1}{\sqrt{\hat{V}^b_{(11)}}}, 
      \dots, 
      \hat{Z}_{h-1}+\frac{\gamma_{h-1}}{\sqrt{\hat{V}^b_{(h-1, h-1)}}}) 
\dconv \NormDist (\frac{\vecf{\gamma}}{\sqrt{ \diag(\matf{V}^b)}},\matf{\Sigma}), 
\end{align}
where $\vecf{\gamma}/\sqrt{ \diag(\matf{V}^b)}$ has element-wise division resulting in $\gamma_x/\sqrt{V^b_{(x,x)}}$ for the $x$th element.
Then, $c^o \equiv \plim \hat{c}^o=(1-\alpha+\beta)$-quantile of $\max_{x \in \mathcal{S}} \{ \NormDist (\vecf{\gamma}/ \sqrt{ \diag(\matf{V}^b)},\matf{\Sigma}) \}$.
Because the mean vector $\vecf{d} / \hat{\vecf{s}} \pconv \vecf{\gamma} / \sqrt{ \diag(\matf{V}^b) }$ and covariance matrix $\hat{\matf{\Sigma}} \pconv \matf{\Sigma}$, then $\NormDist ( \vecf{d}/ \hat{\vecf{s}},\hat{\matf{\Sigma}}) \dconv \NormDist (\vecf{\gamma}/ \sqrt{ \diag(\matf{V}^b) },\matf{\Sigma})$.
By the continuous mapping theorem,
\begin{equation*}
\max_{x \in \mathcal{S}}\{ \NormDist ( \vecf{d} / \hat{\vecf{s}},\hat{\matf{\Sigma}}) \}
\dconv
\max_{x \in \mathcal{S}} \{ \NormDist (\vecf{\gamma} / \sqrt{ \diag(\matf{V}^b) },\matf{\Sigma}) \} .
\end{equation*}

Second, using \cref{eqn:tx-vs-Zx,eqn:t-asy-dist},
\begin{equation}\label{eqn:tmax}
\max_{x \in \mathcal{S}} \hat{t}_x
\dconv
 t^{\text{max}}  \sim \max _{x \in \mathcal{S}}  \bigl\{ \NormDist (\vecf{\gamma} /\sqrt{ \diag(\matf{V}^b) },\matf{\Sigma}) \bigr\}
.
\end{equation}

Third, applying \cref{lemma:convergence}, $\Pr(\max _{x \in \mathcal{S}} \hat{t}_x \ge \hat{c}^o)  
\to \Pr(t^{\text{max}}\ge c^o )$ because $\max_{x \in \mathcal{S}} \hat{t}_x \dconv t^{\text{max}}$ and $\hat{c}^o \pconv c^o$.
In the end, $\Pr(t^{\text{max}}\ge c^o ) \to \alpha - \beta$ because
\begin{equation*}
t^{\text{max}}\sim  \max_{x \in \mathcal{S}}  \bigl\{ \NormDist (\vecf{\gamma} / \sqrt{\diag(\matf{V}^b)},\matf{\Sigma}) \bigr\}
\end{equation*}
and $c^o= [1-(\alpha-\beta)]$-quantile of $\max_{x \in \mathcal{S}}  \bigl\{ \NormDist (\vecf{\gamma}/ \sqrt{\diag(\matf{V}^b)},\matf{\Sigma}) \bigr\}$.

Overall, combining \cref{eqn:left-right,eqn:c-hat-CR-inequality,eqn:left-final}, the asymptotic FWER is bounded above by $(\alpha-\beta)+\beta = \alpha$, as stated in \cref{res:RSW}.
\end{proof}

\begin{lemma}\label{lemma:convergence}
If random variable sequence $X_n$ converges in distribution to a continuous random variable $X$, and random variable sequence $C_n$ converges in probability to constant $c$, then $\Pr(X_n \le C_n) \to \Pr(X \le c)$.
\end{lemma}
\begin{proof}
The core of the proof is to apply Slutsky's Theorem.
Let $X_n$ and $C_n$ be sequences of random variables.
By Slutsky's Theorem, if $X_n$ converges in distribution to a random element $X$ and $C_n$ converges in probability to a constant $c$, then $X_n+C_n \dconv X+c$, and similarly $X_n-C_n \dconv X - c$.
By the definition of convergence in distribution, this means
\begin{equation*}
\Pr(X_n - C_n \le r) \to \Pr(X-c \le r)
\end{equation*}
for any $r$ at which the CDF of $X-c$ is continuous, which was assumed to be all $r$.
At $r=0$,
\begin{equation*}
   \Pr(X_n \le C_n)
=  \Pr(X_n - C_n \le 0)
\to\Pr(X - c \le 0)
= \Pr(X\le c) .
\qedhere
\end{equation*}
\end{proof}

\subsection{Proof of \texorpdfstring{\cref{res:IVQR}}{Theorem \ref{res:IVQR}}}

\begin{proof}
The proof is similar to the proof in \cref{sec:pf-MTP}:
\begin{equation*}
\begin{aligned}
\FWER &\equiv\Pr(\text{reject any true }H_{0j})\\
 &=\Pr \left (\max_{j\colon H_{0j} \text{ true}}\hat{t}_j > c_\alpha \right)\\
 &\le \Pr \left(\max_{j\colon H_{0j}\text{ true}} \hat{Z}_j > c_\alpha \right)\\
 &\le \Pr \Bigl ( \overbrace{\max_{j\in\{1,\ldots,h\}} \hat{Z}_j}^{\equiv\hat{Z}^*}> c_\alpha \Bigr) \\
&\to \alpha .
\end{aligned}
\end{equation*}
The source of each equality and inequality is the same as in the proof in \cref{sec:pf-MTP}.
Only \cref{eqn:txhat-le-Zxhat,eqn:Zhat-asy-dist} are different due to the definitions of $\hat{t}_j$ and $\hat{Z}_j$ differing from before.
These differences are detailed below.

Similar to \cref{eqn:txhat-le-Zxhat}, the first inequality holds because if $H_{0j}$ is true, then $\beta_1 (u_j) \le 0$, so subtracting the true negative value (instead of zero) makes $\hat{Z}_j \ge \hat{t}_j$:
\begin{equation*}
\hat{t}_j \equiv \frac{\hat{\beta}_1(u_j)}{\hat{s}_j} \le \frac{\hat{\beta}_1(u_j)-\beta_1(u_j)}{\hat{s}_j}
\equiv \hat{Z}_j .
\end{equation*}

Similar to \cref{eqn:Zhat-asy-dist}, we need the joint asymptotic distribution of $\hat{\vecf{Z}}\equiv(\hat{Z}_1, \dots, \hat{Z}_{h})$, where 
$\hat{Z}_j\equiv [\hat{\beta}_1(u_j)-\beta_1(u_j)]/\hat{s}_j$ and $\hat{s}_j \equiv \sqrt{\hat{\Omega}_{jj} /n}$.
Gathering these,
\begin{equation*}
\begin{aligned}
\hat{\vecf{Z}}\equiv  
(\hat{Z}_1, \dots, \hat{Z}_{h})
=
\matf{\hat{A}}\sqrt{n}(\vecf{\hat{\beta}}_1-\vecf{\beta}_1) ,
\end{aligned}
\end{equation*}
where $\matf{\hat{A}}$ is a diagonal matrix with elements $\hat{A}_{(jj)}=1/(\sqrt{n}\hat{s}_j)=1/ \sqrt{\hat{\Omega}_{jj}}$, $j=1,\ldots,h$.
Thus,
\begin{equation}
\label{eqn:Zhat-asy-dist-IVQR}
\hat{\vecf{Z}} \dconv \NormDist (\vecf{0},\matf{\Sigma}),
\end{equation}
where $\matf{\Sigma}=\matf{A} \matf{\Omega}  \matf{A}'$.
\end{proof}

\end{appendices}

\end{document}